\documentclass[11pt]{article}

\usepackage{url}
\usepackage{enumerate,mathrsfs,comment}
\usepackage[numbers]{natbib}
\usepackage[OT1]{fontenc}
\usepackage[hypertexnames=false, colorlinks,linkcolor=red,anchorcolor=blue,citecolor=blue]{hyperref}
\usepackage{fullpage}

\usepackage[protrusion=false,expansion=true]{microtype}

\RequirePackage{amsmath}
\RequirePackage{amssymb}
\RequirePackage{amsthm}
\RequirePackage{bm}
\RequirePackage{url}
\usepackage{natbib}
\usepackage{multirow}
\usepackage{graphicx}
\usepackage{subfigure}
\usepackage{makecell}
\usepackage{booktabs}
\usepackage{array}
\usepackage{url}
\usepackage{algorithm}
\usepackage{algorithmic}
\usepackage{dsfont}

\let\tilde\widetilde

\newcommand{\cH}{\mathcal{H}}

\newcommand{\cX}{\mathcal{X}}

\newcommand{\EE}{\mathbb{E}}

\newcommand{\PP}{\mathbb{P}}

\newcommand{\argmin}{\mathop{\mathrm{argmin}}}
\newcommand{\argmax}{\mathop{\mathrm{argmax}}}

\DeclareMathOperator{\Var}{{\rm Var}}

\newcommand{\la}{\langle}
\newcommand{\ra}{\rangle}

\DeclareMathOperator{\E}{E}

\newtheorem{proposition}{\indent \bf Proposition}
\numberwithin{equation}{section}

\def\bbR{\mathbb{R}}

\theoremstyle{plain}
\numberwithin{equation}{section}
\newtheorem{Theorem}{Theorem}[section]
\newtheorem{Lemma}[Theorem]{Lemma}
\newtheorem{Corollary}[Theorem]{Corollary}
\newtheorem{Proposition}[Theorem]{Proposition}
\newtheorem{Assumption}{Assumption}

\newtheorem{lemma}{\indent \bf Lemma}
\newtheorem{Remark}{Remark}[section]

\newtheoremstyle{mytheoremstyle} 
    {\topsep}                    
    {\topsep}                    
    {\normalfont}                   
    {}                           
    {\bfseries}                   
    {.}                          
    {.5em}                       
    {}  

\theoremstyle{mytheoremstyle}

\ifx\BlackBox\undefined
\newcommand{\BlackBox}{\rule{1.5ex}{1.5ex}}  
\fi

\ifx\QED\undefined
\def\QED{~\rule[-1pt]{5pt}{5pt}\par\medskip}
\fi

\ifx\proof\undefined
\newenvironment{proof}{\par\noindent{\bf Proof\ }}{\hfill\BlackBox\\[2mm]}
\fi

\ifx\theorem\undefined
\newtheorem{theorem}{Theorem}
\fi
\ifx\example\undefined
\newtheorem{example}[theorem]{Example}
\fi
\ifx\property\undefined
\newtheorem{property}{Property}
\fi
\ifx\lemma\undefined
\newtheorem{lemma}[theorem]{Lemma}
\fi
\ifx\proposition\undefined
\newtheorem{proposition}[theorem]{Proposition}
\fi
\ifx\remark\undefined
\newtheorem{remark}[theorem]{Remark}
\fi
\ifx\corollary\undefined
\newtheorem{corollary}[theorem]{Corollary}
\fi
\ifx\definition\undefined
\newtheorem{definition}[theorem]{Definition}
\fi
\ifx\conjecture\undefined
\newtheorem{conjecture}[theorem]{Conjecture}
\fi
\ifx\fact\undefined
\newtheorem{fact}[theorem]{Fact}
\fi
\ifx\claim\undefined
\newtheorem{claim}[theorem]{Claim}
\fi
\ifx\assumption\undefined

\fi
\numberwithin{equation}{section}
\numberwithin{theorem}{section}
\usepackage{multirow}

\usepackage{pifont}
\newcommand{\cmark}{\ding{51}}
\newcommand{\xmark}{\ding{55}}

\long\def\comment#1{}

\begin{document}

\title{How Many Machines Can We Use in Parallel Computing for Kernel Ridge Regression?}

\author
{
    Meimei Liu\thanks{Postdoc, Department of Statistical Science, Duke University, Durham, NC 27705. E-mail: meimei.liu@duke.edu.},
    Zuofeng Shang\thanks{Assistant Professor, Department of Mathematical Sciences, Indiana University-Purdue University at Indianapolis. Email: shangzf@iu.edu. Research sponsored by NSF DMS-1764280 and a startup grant from IUPUI.},
    Guang Cheng\thanks{Professor, Department of Statistics, Purdue University, West Lafayette, IN 47906. E-mail: chengg@purdue.edu. Research Sponsored by NSF CAREER Award DMS-1151692, DMS-1712907, and Office of Naval Research (ONR N00014-15-1-2331).}
}

\date{}
\maketitle

\begin{abstract}
  This paper aims to solve a basic problem in distributed statistical inference: how many machines can we use in parallel computing? In kernel ridge regression, we address this question in two important settings: nonparametric estimation and hypothesis testing. Specifically, we find a range for the number of machines under which optimal estimation/testing is  achievable. The employed empirical processes method provides a unified framework, that allows us to handle various regression problems (such as thin-plate splines and nonparametric additive regression) under different settings (such as univariate, multivariate and diverging-dimensional designs).  It is worth noting that the upper bounds of the number of machines are proven to be un-improvable (upto a logarithmic factor) in two important cases: smoothing spline regression and Gaussian RKHS regression. Our theoretical findings are backed by thorough numerical studies.
\end{abstract}


\section{Introduction}

In the parallel computing environment, a common practice is to distribute a massive dataset to multiple processors, and then aggregate local results obtained from separate machines into global counterparts. This Divide-and-Conquer (D\&C) strategy often requires a growing number of machines to deal with an increasingly large dataset. An important question to statisticians is "how many machines can we use in parallel computing to guarantee statistical optimality?" The present work aims to explore this basic yet fundamentally important question in a classical nonparametric regression setup, i.e., kernel ridge regression (KRR). This can be done by carefully analyzing statistical versus computational trade-off in the D\&C framework, where the number of deployed machines is treated as a simple proxy for computing cost.

Recently, researchers have made impressive progress about KRR in the modern D\&C framework with different conquer strategies; examples include median-of-means estimator proposed by \cite{minsker2017distributed}, Bayesian aggregation considered by \cite{shang2015bayesian,weberbayesian,srivastava2015scalable,szabo2017asymptotic}, and simple averaging considered by \cite{zhang2013divide} and \cite{shang2017computational}. 
Upper bounds for the number of machines $s$ have been studied in such strategies to guarantee good property. For instance, \cite{zhang2013divide} showed that, when $s$ processors are employed with $s$ in a suitable range, D$\&$C method still preserves minimax optimal estimation. In smoothing spline regression  (a special case of KRR), \cite{shang2017computational} derived \textit{critical}, i.e., un-improvable, upper bounds for $s$ to achieve either optimal estimation or optimal testing, but their results are only valid in univariate fixed design. The critical bound for estimation obtained by \cite{shang2017computational} significantly improves the one given in \cite{zhang2013divide}. Nonetheless, it remains unknown if results obtained in \cite{shang2017computational} continues to hold in a more general KRR framework where the design is either multivariate or random. On the other hand, there is a lack of literature dealing with nonparametric testing in general KRR. To the best of our knowledge, \cite{shang2017computational} is the only reference but in the special smoothing spline regression with univariate fixed designs.

In this paper, we consider KRR in the D$\&$C regime in a general setup: design is random and multivariate. As our technical contribution, we characterize the upper bounds of $s$ for achieving optimal estimation and testing based on quantifying an empirical process (EP), such that a sharper concentration bound of the EP leads to a tighter upper bound of $s$. Our EP approaches can handle various function spaces including Sobolev space, Gaussian RKHS, or spaces of special structures such as additive functions, in a unified manner. As an illustration example, in the particular smoothing spline regression, we introduce the Green function for equivalent kernels to the EP bound and achieve a polynomial order improvement of $s$ compared with \cite{zhang2013divide}. It is worthy noting that our upper bound is almost identical as \cite{shang2017computational} (upto a logarithmic factor) for optimal estimation, which is proven to be un-improvable. 

The second main contribution of this paper is to propose a Wald type test statistic for nonparametric testing in D\&C regime. Asymptotic null distribution and
power behaviors of the proposed test statistic are carefully analyzed. One important finding is that the upper bounds of $s$ for optimal testing 
are dramatically different from estimation, indicating the essentially different natures of the two problems. 
Our testing results are derived in a general framework that cover the aforementioned important function spaces.
As an important byproduct, we derive a minimax rate of testing for nonparametric additive models
with diverging number of components which is new in literature. Such rate is crucial in deriving the upper bound for $s$ for optimal testing,
and is of independent interest.

\section{Background and Distributed Kernel Ridge Regression}
We begin by introducing some background on reproducing kernel Hilbert space (RKHS), and our nonparametric testing formulation under the distributed kernel ridge regression. 
 
\subsection{Nonparametric regression in reproducing kernel Hilbert spaces}
Suppose that data $\{(Y_i,X_i): i=1,\ldots,N\}$ are \textit{iid} generated from the following regression model
\begin{equation}\label{krr:model}
Y_i=f(X_i)+\epsilon_i,\,\,i=1,\ldots,N,
\end{equation}
where $\epsilon_i$ are random errors
with $E\{\epsilon_i\}=0$, $E\{\epsilon_i^2|X_i)=\sigma^2(X_i)>0$, 
the covariates $X_i\in \mathcal{X} \subseteq \bbR^d$ follows a distribution 
$\pi(x)$, and $Y_i \in \mathbb{R}$ is a real-valued response. 
Here, $d\geq 1$ is either fixed or diverging with $N$, and $f$ is unknown. 

Throughout we assume that
$f\in\cH$, where  $\mathcal{H} \subset L_\pi^2(\mathcal{X})$ is a reproducing kernel Hilbert
space (RKHS) associated with an inner product
$\langle \cdot,\cdot \rangle_{\cH}$ and a reproducing kernel function $R(\cdot,\cdot):
\mathcal{X}\times\mathcal{X}\to\bbR$. 
By Mercer's Theorem, $R$ has the following spectral expansion (\cite{wahba1990spline}):
\begin{equation*}
R(x,x') = \sum_{i=1}^\infty \mu_i \varphi_i(x)\varphi_i(x'),
\,\,\,\,x,x'\in\mathcal{X},
\end{equation*}
where $\mu_1 \geq \mu_2 \geq \cdots \geq 0$ 
is a sequence of eigenvalues and $\{\varphi_{i}\}_{i=1}^{\infty}$ form a
basis in $L_\pi^2(\mathcal{X})$. Moreover, for any $i,j \in \mathbb{N}$,
$$
\langle \varphi_{i}, \varphi_{j}\rangle_{L_\pi^2(\mathcal{X})} = \delta_{ij} \quad \quad
\text{and} \quad \quad \langle \varphi_{i}, \varphi_{j}\rangle_{\cH} = \delta_{ij}/\mu_i,
$$
where $\delta_{ij}$ is Kronecker's $\delta$. 

We introduce a norm $\| \cdot \|$ in $\cH$ by combining the $L_2$ norm and $\|\cdot\|_\cH$ norm to facilitate our statistical inference theory. For $f,g\in\mathcal{H}$, define 
\begin{equation}\label{eq:new_norm}
\langle f,g\rangle=V(f,g)+\lambda\langle f,g\rangle_{\mathcal{H}},
\end{equation}
where $V(f,g)=E\{f(X)g(X)\}$ and $\lambda>0$ is the penalization parameter. Clearly, $\langle\cdot,\cdot\rangle$ defines an inner product on $\mathcal{H}$. It is easy to prove that $(\mathcal{H},\langle\cdot,\cdot\rangle)$ is also a RKHS with reproducing kernel function  $K(\cdot, \cdot)$ satisfying 
the following so-called reproducing property: 
 \begin{equation*}
   \la f, K_x(\cdot)\ra = f(x), \;\; \text{for all } f\in \cH,
 \end{equation*}
where  $K_x(\cdot)=K(x,\cdot)$ for $x\in\mathcal{X}$.

For any $f\in \cH$, we can express the function in terms of the Fourier expansion as $f=\sum_{\nu\ge1}V(f,\varphi_\nu)\varphi_\nu$. Therefore, 
\begin{equation}\label{eq:expansion}
\la f, \varphi_\nu \ra = \sum_{i\ge1}V(f,\varphi_i) \la \varphi_i, \varphi_\nu \ra = V(f, \varphi_\nu)(1+ \lambda/\mu_\nu). 
\end{equation}
Replacing $f$ with $K_x$ in (\ref{eq:expansion}), we have $V(K_x, \varphi_\nu) = \frac{\la K_x, \varphi_\nu \ra}{1+ \lambda/\mu_\nu} = \frac{\varphi_\nu(x)}{1+ \lambda/\mu_\nu}$. 
Then for any $x,y\in\mathcal{X}$, $K(x,y)$ 
has an explicit eigen-expansion expressed as
$$K(x,y)=\sum_{\nu\geq 1} V(K_x, \varphi_\nu)\varphi_\nu(y) = \sum_{\nu\ge1}\frac{\varphi_\nu(x)\varphi_\nu(y)}{1+\lambda/\mu_\nu}.$$

\subsection{Distributed kernel ridge regression}
For estimating $f$, we consider the kernel ridge regression (KRR) in
a divide-and-conquer (D$\&$C) regime. First, randomly divide the $N$ samples into $s$ subsamples. Let $I_j$ denote the set of indices of the observations from subsample $j$ for $j=1,\ldots,s$. For simplicity, suppose $|I_j|=n$, i.e., all subsamples are of equal sizes. Hence, the total sample size is $N=ns$. Then, we estimate $f$ based on subsample $j$ through the following KRR method:
\[
\widehat{f}_j=\argmin_{f\in\mathcal{H}}\ell_{j,\lambda}(f)
\equiv\argmin_{f\in\mathcal{H}}\frac{1}{2n}\sum_{i\in I_j}(Y_i-f(X_i))^2+\frac{\lambda}{2}\|f\|_{\mathcal{H}}^2,\,\,\,\,j=1,\ldots,s,
\]
where $\lambda>0$ is the penalization parameter.
The D$\&$C estimator of $f$ is defined as the average of $\widehat{f}_j$'s, that is, 
$\bar{f}=\sum_{j=1}^s\widehat{f}_j/s$.

Based on $\bar{f}$, we further propose a Wald-type statistic
$T_{N,\lambda}:=\|\bar{f}\|^2$ for testing the hypothesis
\begin{equation}\label{H0:H1}
\textrm{$H_0: f=0$, vs. $H_1: f\in\mathcal{H}\backslash\{0\}$.}
\end{equation} 
In general, testing $f = f_0$ (for a known $f_0$) is equivalent to testing $f_\ast\equiv f-f_0=0$. So, (\ref{H0:H1}) has no
loss of generality. 
\section{Main results}
In this section, we derive some general results relating to $\bar{f}$ and $T_{N,\lambda}$.
Let us first introduce some regularity assumptions.
\subsection{Assumptions}
The following Assumptions \ref{A1} and \ref{A2} require that the design density is bounded and 
the error $\epsilon$ has finite fourth moment, which are commonly used in literature, see \cite{eggermont2009maximum}. 
\begin{Assumption}\label{A1}
 There exists a constant $ c_\pi >0$ such that for all $x\in\mathcal{X}$,
 $0\le \pi(x), \sigma^2(x)\le  c_\pi$.
\end{Assumption}

\begin{Assumption}\label{A2}
There exists a positive constant $\tau$ such that $E\{\epsilon^4|X\}<\tau$ almost surely.
\end{Assumption}
Define $\|f\|_{\sup}=\sup_{x\in\mathcal{X}}|f(x)|$ as the supremum norm of $f$. We further assume that $\{\varphi_\nu\}_{\nu=1}^\infty$ are uniformly bounded on $\cX$, and $\{\mu_{\nu}\}_{\nu=1}^\infty$ satisfy certain tail sum property.

\begin{Assumption}\label{A3}
$c_\varphi :=\sup_{j\ge 1} \|\varphi_j\|_{\sup} < \infty $ and $\sup_{k\geq 1} \frac{\sum_{\nu=k+1}^\infty \mu_i}{k \mu_k} < \infty$. 
\end{Assumption} 

The uniform boundedness condition of eigen-functions holds for various kernels, example includes univariate periodic kernel, 2-dimensional Gaussian kernel, multivariate additive kernel; see \cite{lu2016nonparametric}, \cite{minh2006mercer} and reference therein. The tail sum property can also be verified in various RKHS, and is deferred to the Appendix.

Define $h^{-1}:=\sum_{\nu\ge1}\frac{1}{1+\lambda/\mu_\nu}$ as effective dimension.  It has been widely studied in reference \cite{bartlett2005local}, \cite{mendelson2002geometric}, \cite{zhang2005learning} etc. There is an explicit relationship between $h$ and $\lambda$ as illustrated in various concrete examples in Section \ref{sec:two:examples}. Another quantity of interest is the series $\sum_{\nu\ge1}(1+\lambda/\mu_\nu)^{-2}$, which represents the variance term defined in Theorem \ref{basic:thm:testing:null:distr}. In the following Proposition \ref{prop:var_eff}, we show that such variance term has the same order of $h^{-1}$. 

\begin{Proposition}\label{prop:var_eff}
Suppose Assumption \ref{A3} holds. For any $\lambda>0$, $\sum_{\nu\ge1}(1+\lambda/\mu_\nu)^{-2}\asymp h^{-1}$.  
\end{Proposition}

Define $Pf=E_X\{f(X)\}$, $P_j f=n^{-1}\sum_{i\in I_j} f(X_i)$ and 
\[
\xi_j=\sup_{\substack{f,g\in\mathcal{H}\\ \|f\|=\|g\|=1}}|P_jfg-Pfg|,\,\, 
1\le j\le s.
\]
Here, $\xi_j$ is the supremum of the empirical processes based on subsample $j$. The quantity $\max_{1\leq j\leq s}\xi_j$ plays a vital role in determining the critical upper bound of $s$ to guarantee statistical optimality. As shown in our main theorems, a sharper bound of $\xi_j$ directly leads to an improved upper bound of $s$. Assumption \ref{A4} provides a concentration bound for $\xi_j$, and says that $\xi_j$ are uniformly bounded by $\sqrt{\frac{\log^b{N}}{nh^a}}$, 
$a,b$ are constants 
that are specified in various kernels. Verification of Assumption \ref{A4} is deferred to
Section \ref{sec:two:examples} in concrete settings based on empirical processes methods, where the values of $a,b$ will be explicitly specified.
\begin{Assumption}\label{A4} There exist nonnegative constants $a,b$ such that
$$\max_{1\le j\le s}\xi_j=O_P\left(\sqrt{\frac{\log^b{N}}{nh^a}}\right).$$
\end{Assumption}

\subsection{Minimax optimal estimation}\label{sec:est}
In this section, we derive a general error bound for $\bar{f}$.
Let $\textbf{X}_j=\{X_i: i\in I_j\}$ and $\textbf{X}=\{\textbf{X}_1,\ldots,\textbf{X}_s\}$. 
Suppose that (\ref{krr:model}) holds under $f=f_0$. 
For convenience, let $\mathcal{P}_\lambda$ be a self-adjoint operator from $\mathcal{H}$
to itself such that $\langle \mathcal{P}_\lambda f,g\rangle=\lambda\langle f,g\rangle_{\mathcal{H}}$
for all $f,g\in\mathcal{H}$. The existence of $\mathcal{P}_\lambda$ follows by \cite[Proposition 2.1]{shang2013local}. We first 
obtain a uniform error bound for $\widehat{f}_j$'s in the following Lemma \ref{basic:thm:estimation}.
\begin{Lemma}\label{basic:thm:estimation}
Suppose Assumptions \ref{A1},\ref{A3},\ref{A4} are satisfied
and 
$\log^b{N}=o(nh^a)$ with $a,b$ given in Assumption \ref{A4}.
Then with probability approaching one,
for any $1\le j\le s$,
\begin{eqnarray}
E\{\|\widehat{f}_j-E\{\widehat{f}_j|\textbf{X}_j\}-\frac{1}{n}\sum_{i\in I_j}\epsilon_i K_{X_i}\|^2
|\textbf{X}_j\}&\le& \frac{4 c_\pi  c_\varphi^2\xi_j^2}{nh},\label{basic:thm:estimation:conc:-1}\\
\|E\{\widehat{f}_j|\textbf{X}_j\}-f_0+\mathcal{P}_\lambda f_0\|
&\le& 2\xi_j\lambda^{1/2}\|f_0\|_{\mathcal{H}}
\label{basic:thm:estimation:conc:0}
\end{eqnarray}

\end{Lemma}
(\ref{basic:thm:estimation:conc:-1}) quantifies the deviation
from $\widehat{f}_j$ to its conditional mean through a higher order remainder term, 
and (\ref{basic:thm:estimation:conc:0}) quantifies the bias of $\widehat{f}_j$. 
Lemma \ref{basic:thm:estimation} immediately leads to the following result on $\bar{f}$. 
Specifically, (\ref{basic:thm:estimation:conc:-1}) and (\ref{basic:thm:estimation:conc:0})
lead to (\ref{cor:thm1:eqn:0}), which, together with the rates of $\sum_{i=1}^N\epsilon_i K_{X_i}$ and $\mathcal{P}_\lambda f_0$ in Lemma \ref{le:prelim}, leads to (\ref{cor:thm1:eqn:1}).
\begin{Theorem}\label{thm:of:thm1}
If the conditions in Lemma \ref{basic:thm:estimation} hold,
then with probability approaching one, 
\begin{eqnarray}
E\{\|\bar{f}-\frac{1}{N}\sum_{i=1}^N\epsilon_i K_{X_i}-f_0+\mathcal{P}_\lambda f_0\|^2|\textbf{X}\}
&\le& 4\left(\frac{ c_\pi  c_\varphi^2}{Nh}+\lambda \|f_0\|^2_{\mathcal{H}}\right)\max_{1\le j\le s}\xi_j^2,
\label{cor:thm1:eqn:0}
\end{eqnarray}
\begin{eqnarray}
E\{\|\bar{f}-f_0\|^2|\textbf{X}\}&\le&\frac{4 c_\pi  c_{\varphi}^2}{Nh}+8\lambda \|f_0\|^2_{\mathcal{H}}.
\label{cor:thm1:eqn:1}
\end{eqnarray}
\end{Theorem}
Theorem \ref{thm:of:thm1} is a general result that holds 
for many commonly used kernels. Note that $n=N/s$, the condition $\log^b N = o(nh^a)$  directly implies that as long as $s$ is dominated by $Nh^a/\log^b N$, the conditional mean squared errors can be upper bounded by the variance term $(Nh)^{-1}$ and the squared bias term $\lambda\|f_0\|^2_\mathcal{H}$. 
 Then the minimax optimal estimation can be obtained through the particular $\lambda$ that satisfies such bias-variance trade-off; see \cite{bartlett2005local}, \cite{yang2015randomized}. 
 Section \ref{sec:two:examples} further illustrates concrete and interpretable guarantees on the conditional mean squared errors to particular kernels. 

It is worthy to note that, through the condition of Lemma \ref{basic:thm:estimation} and Theorem \ref{thm:of:thm1}, we build a direct connection between the upper bound of $s$ and the uniform bound of the empirical process $\xi_j$. That is, a tighter upper bound of $s$ can be achieved by a sharper concentration bound of $\max_{1\leq j \leq s} \xi_j$, which is guaranteed by the empirical process methods in this work. For instance, in Section \ref{eg:example1} the smoothing spline regression, we introduce the Green function for equivalent kernels in \cite{eggermont2009maximum} to provide a sharp concentration bound of $\xi_j$ with $a=b=1$. Consequently, we achieve an upper bound for $s$ almost identical to the critical one obtained by \cite{shang2017computational} (upto a logarithmic factor), and improve the one obtained by \cite{zhang2013divide} in polynomial order.

\subsection{Minimax optimal testing}\label{sec:general:testing}
In this section, we derive the asymptotic distribution of
$T_{N,\lambda}:=\|\bar{f}\|^2$ and further investigate its power behavior.
For simplicity, assume that $\sigma^2(x)\equiv\sigma^2$
is known. Otherwise, we can replace $\sigma^2$ by 
its consistent estimator to fulfill our procedure.
We will show that the distributed test statistic $T_{N,\lambda}$ can achieve minimax rate of testing (MRT), provided that the number of divisions $s$ belongs to a suitable range. 
Here, MRT is defined as the minimal distance between the null and the alternative 
 hypotheses such that valid testing is possible.  The range of $s$ is determined based on the criteria that the proposed test statistic can asymptotically achieve correct size and high power.

Before proving consistency of the test statistics $T_{N,\lambda}$, i.e.,
Theorem \ref{basic:thm:testing:null:distr}, let us state a technical lemma.
Define $W(N)=\sum_{1\le i<k\le N}W_{ik}$ with $W_{ik}=2\epsilon_i\epsilon_k K(X_i,X_k)$,
and let $\sigma^2(N)=\textrm{Var}(W(N))$.
Define the empirical kernel matrix   $\textbf{K}=[K(X_i,X_j)]_{i,j=1}^N$ and $\boldsymbol{\epsilon}=(\epsilon_1,\ldots,\epsilon_N)^T$.
\begin{Lemma}\label{basic:lemma:for:testing}
Suppose Assumptions \ref{A1},\ref{A2}, \ref{A3}, \ref{A4} are all satisfied, and $N\to\infty$,
$h=o(1)$, $Nh^2 \to \infty$. 
Then it holds that 
\begin{equation}\label{basic:lemma:for:testing:eqn:0}
\boldsymbol{\epsilon}'\textbf{K}\boldsymbol{\epsilon}=\sigma^2Nh^{-1}+W(N)+O_P(\sqrt{Nh^{-2}}).
\end{equation}
Furthermore, as $N\to\infty$, $\frac{W(N)}{\sigma(N)}\overset{d}{\longrightarrow}N(0,1)$, where 
$\sigma^2(N)=
2\sigma^4 N(N-1)\sum_{\nu\ge 1}\frac{1}{(1+\lambda/\mu_\nu)^2}\asymp N^2 h^{-1}$. 
\end{Lemma}  
The following theorem shows that $T_{N,\lambda}$ is asymptotically normal
under $H_0$. The key condition to obtain such a result is
$\log^b N = o(nh^{a+1})$, where $a,b$ are determined through the uniform bound of $\xi_j$ in Assumption \ref{A4}. 
This condition in turn leads to upper bounds for $s$ to achieve MRT; see Section \ref{sec:two:examples} for detailed illustrations.
\begin{Theorem}\label{basic:thm:testing:null:distr}
Suppose Assumptions \ref{A1} to \ref{A4} are all satisfied, and as $N\to\infty$,
$h=o(1)$, $Nh^2 \to \infty$, and $\log^b N = o(nh^{a+1})$. 
Then, as $N\to\infty$,
\[
\frac{N^2}{\sigma(N)}\left(T_{N,\lambda}-\frac{\sigma^2}{Nh}\right)
\overset{d}{\longrightarrow}N(0,1).
\]
\end{Theorem}

By Theorem \ref{basic:thm:testing:null:distr}, we can define an
asymptotic testing rule with $(1-\alpha)$ significance level as follows:
\[
\psi_{N,\lambda}=I\left(|T_{N,\lambda}-\sigma^2/(Nh)|\ge z_{1-\alpha/2}\sigma(N)/N^2\right),
\]
where $z_{1-\alpha/2}$ is the $(1-\alpha/2)\times 100$ percentile of standard normal distribution. 

For any $f\in \mathcal{H}$, define 
\begin{align*}
b_{N,\lambda}& =(\lambda^{1/2}\|f\|_\mathcal{H}+(Nh)^{-1/2})\sqrt{\frac{\log^b{N}}{nh^a}},\,\,\,\,\textrm{and}\\
d_{N,\lambda}& =\lambda^{1/2}\|f\|_\mathcal{H}+(Nh^{1/2})^{-1/2}+N^{-1/2}+ b_{N,\lambda}^{1/2}(Nh)^{-1/4}+b_{N,\lambda}.
\end{align*}

$d_{N,\lambda}$ is used to measure the distance between the null and the alternative hypotheses. 
The following Theorem \ref{thm:power:random:design} shows that, if the alternative signal $f$ is separated from zero by an order $d_{N,\lambda}$, then the proposed test statistic asymptotically achieves
 high power. To achieve optimal testing, it is sufficient to minimize $d_{N,\lambda}$. As long as $s$ is dominated by $(Nh^{a+1}/\log^b N)$, $d_{N,\lambda}$ can be simplified as 

\begin{equation}\label{eq:trade_off:testing}
d_{N,\lambda}\asymp\quad \underbrace{\lambda^{1/2}\|f\|_\mathcal{H}}_{\textrm{Bias of }\bar{f}} \quad + \underbrace{(Nh^{1/2})^{-1/2}}_{\textrm{Standard deviation of }T_{N,\lambda}}
\end{equation}
Then, MRT can be achieved by selecting $\lambda$ to balance the tradeoff between the bias of $\bar{f}$ and the standard derivation of $T_{N,\lambda}$; see \cite{ingster1993asymptotically}, \cite{wei2017local}. It is worth noting that, such a tradeoff in (\ref{eq:trade_off:testing}) for testing is different from the bias-variance tradeoff in (\ref{thm:of:thm1}) for estimation, thus leading to different optimal testing rate. 

\begin{Theorem}\label{thm:power:random:design}
If the conditions in Theorem \ref{basic:thm:testing:null:distr} hold, then
for any $\varepsilon>0$, there exist $C_\varepsilon$ and $N_\varepsilon$ s.t.
\[
\inf_{\substack{
\|f\|\ge C_\varepsilon d_{N,\lambda}}} P_f(\psi_{N,\lambda}=1)\ge 1-\varepsilon,
\,\,\,\,\textrm{for any $N\ge N_\varepsilon$.}
\]
\end{Theorem}

Section \ref{sec:two:examples} will develop upper bounds for $s$ 
in various concrete examples based on the above general theorems. 
Our results will indicate that the ranges for $s$ to achieve MRT are dramatically different from
ones to achieve optimal estimation.

\subsection{Examples}\label{sec:two:examples}
In this section, we derive upper bounds for $s$ in four featured examples to achieve optimal estimation/testing,
based on the general results obtained in Sections \ref{sec:est} and \ref{sec:general:testing}.
Our examples cover the settings of univariate, multivariate and diverging-dimensional designs.

\subsubsection{Example 1: Smoothing spline regression}\label{eg:example1}

Suppose $\mathcal{H} = \{f\in S^m(\mathbb{I}): \|f\|_{\mathcal{H}}\leq C\}$ for a constant $C>0$, where 
$S^m(\mathbb{I})$ is the $m$th order Sobolev space on $\mathbb{I}\equiv [0,1]$, i.e.,
\begin{align*}
S^m(\mathbb{I})=\big\{f\in L^2(\mathbb{I})|\, &f^{(j)} \mbox{ are abs. cont. for}\;  j=0,1,\ldots,m-1,\;
\\
&\mbox{and}\; \int_{\mathbb{I}} |f^{(m)}(x)|^2dx<\infty\big\},
\end{align*}
and $\|f\|_{\mathcal{H}}=\int_{\mathbb{I}} |f^{(m)}(x)|^2dx$. 
Then model (\ref{krr:model}) becomes the usual smoothing spline regression.
In addition to Assumption \ref{A1}, we assume that 
\begin{equation}\label{quasi:uniform:assumption}
 c_\pi ^{-1}\le\pi(x)\le  c_\pi ,\,\,\textrm{for any $x\in\mathbb{I}$.}
\end{equation}
We call the design satisfying (\ref{quasi:uniform:assumption}) as quasi-uniform, a common assumption on many statistical problems; see \cite{eggermont2009maximum}. Quasi-uniform assumption excludes cases where design density is (nearly) zero at certain data points, which may cause estimation inaccuracy at those points. 

It is known that when $m>1/2$, $S^m(\mathbb{I})$ is a RKHS under the inner product $\langle\cdot,\cdot\rangle$; see \cite{shang2013local}, \cite{fan2001generalized}. Meanwhile, Assumption \ref{A3}
holds with kernel eigenvalues $\mu_\nu\asymp \nu^{-2m}$, $\nu\ge1$.
Hence, Proposition \ref{prop:var_eff} holds with $h\asymp \lambda^{1/(2m)}$. 
We next provide a sharp concentration inequality to bound $\xi_j$. 
\begin{Proposition}\label{prop:EL09}
Under (\ref{quasi:uniform:assumption}),
there exist universal positive constants $c_1,c_2,c_3$ such that for any $1\le j\le s$,
\[
P\left(\xi_j\ge t\right)\le 2n\exp\left(-\frac{nht^2}{c_1+c_2t}\right),\,\,
\textrm{for all $t\ge c_3(nh)^{-1}$.}
\]
\end{Proposition}
The proof of Proposition \ref{prop:EL09} is based on the novel technical tool that we introduce into D$\&$C framework:  the Green function for equivalent kernels; see \cite[Corollary 5.41]{eggermont2009maximum}. An immediate consequence of Proposition \ref{prop:EL09} is that
 Assumption \ref{A4} holds with $a=b=1$. Then based on Theorem \ref{thm:of:thm1} and Theorem \ref{thm:power:random:design}, we have the following results. 

\begin{Corollary}\label{smoothing:spline:cor}
Suppose that $\mathcal{H}=S^m(\mathbb{I})$, (\ref{quasi:uniform:assumption}), 
Assumptions \ref{A1} and \ref{A2} hold. 
\begin{enumerate}
\item If $m>1/2$, $s=o(N^{2m/(2m+1)}/\log{N})$ and $\lambda\asymp N^{-2m/(2m+1)}$,
then $\|\bar{f}-f_0\|=O_P(N^{-m/(2m+1)})$.
\item If $m>3/4$, $s=o(N^{(4m-3)/(4m+1)}/\log{N})$ and $\lambda\asymp N^{-4m/(4m+1)}$,
then the Wald-type test achieves minimax rate of testing $N^{-2m/(4m+1)}$.
\end{enumerate}
\end{Corollary}
It is known that the estimation rate $N^{-m/(2m+1)}$ is minimax-optimal; see \cite{stone1985additive}. Furthermore, the testing rate $N^{-2m/(4m+1)}$ is also minimax optimal, in the sense of \cite{ingster1993asymptotically}. It is worth noting that the upper bound for $s=o(N^{2m/(2m+1)}/\log{N})$ matches (upto a logarithmic factor) 
the critical one 
by \cite{shang2017computational} in evenly spaced design,
which is substantially larger than the one obtained by \cite{zhang2013divide}, i.e., $s=o(N^{(2m-1)/(2m+1)}/\log{N})$ for bounded eigenfunctions; see Table \ref{table:est_compare} for the comparison. 
\begin{table}[h]
\centering
\renewcommand{\arraystretch}{1.3} 
\label{table:est_compare}
\begin{tabular}{c|c|c|c}
\hline
                                         & Zhang et al \cite{zhang2013divide}                                   & Shang et al \cite{shang2017computational}    & Our approach                \\ \hline
          {\begin{tabular}[c]{@{}c@{}}  smoothing spline\\regression \end{tabular}}                              & \begin{tabular}[c]{@{}c@{}} $s\lesssim N^{\frac{2m-1}{2m+1}}/\log N$\\sharpness of $s$ \xmark \end{tabular}
         & \begin{tabular}[c]{@{}c@{}} $s\lesssim N^{\frac{2m}{2m+1}}$\\sharpness of $s$ \cmark\end{tabular}   &\begin{tabular}[c]{@{}c@{}}$s=o(N^{\frac{2m}{2m+1}}/\log N)$\\sharpness of $s$ \cmark  \end{tabular}\\ \hline
\end{tabular}
\caption{Comparison of upper bounds of $s$ to achieve minimax optimal estimation.}
\end{table}

\subsubsection{Example 2: Nonparametric additive regression}\label{eg:example4}
Consider the function space
\[
\mathcal{H}= \{f(x_1,\ldots,x_d)=\sum_{k=1}^d f_k(x_k): \, f_k \in S^m(\mathbb{I}), \; \|f_k\|_\mathcal{H} \leq C\; \text{for}\;  k=1,\ldots,d\},
\]
where $C>0$ is a constant. That is, any $f\in\mathcal{H}$ has an additive decomposition of $f_k$'s. Here, $d$ is either fixed or slowly diverging. Such additive model has been well studied in many literatures; see \cite{stone1985additive}, \cite{meier2009high}, \cite{raskutti2012minimax}, \cite{yuan2016minimax} among others. 
 For $x= (x_1, \cdots, x_d) \in \mathcal{X}$, suppose $x_i, x_j$ are independent for $i\neq j \in \{1, \cdots, d\}$ and each $x_i$ satisfies (\ref{quasi:uniform:assumption}). For identifiability, assume 
 $E\{f_k(x_k)\}= 0$ for all $1\le k\le d$. 
For $f=\sum_{k=1}^d f_k$ and $g=\sum_{k=1}^d g_k$, define
\begin{align*}
 \langle f,g\rangle_{\mathcal{H}} & = \sum_{k=1}^d \langle f_k, g_k \rangle_{\mathcal{H}} = \sum_{k=1}^d \int_{\mathbb{I}}f_k^{(m)}(x)g_k^{(m)}(x)dx,\quad  \text{and} \\
  V(f,g) & = \sum_{k=1}^d V_k(f_k,g_k)\equiv\sum_{k=1}^d
E\{f_k(X_k)g_k(X_k)\}.
\end{align*}

It is easy to verify that $\mathcal{H}$ is an RKHS under $\langle\cdot,\cdot\rangle$ defined in (\ref{eq:new_norm}). Lemma \ref{le:add:prelim} below summarizes the properties for the $\mathcal{H}$ with $d$ additive components.

\begin{Lemma}\label{le:add:prelim}
\begin{enumerate}
\item There exist eigenfunctions $\varphi_\nu$ and eigenvalues $\mu_\nu$ that satisfying Assumption \ref{A3}. \label{le:add:prelim:a}
\item It holds that $\sum_{\nu\ge1}(1+\lambda/\mu_\nu)^{-1} := h^{-1} \asymp d\lambda^{-1/(2m)}$, and  $\sum_{\nu\ge1}(1+\lambda/\mu_\nu)^{-2} \asymp h^{-1}$ accordingly. \label{le:add:prelim:b}
\item For $f\in \mathcal{H}$, $\|\mathcal{P}_\lambda f\|^2 \leq cd\lambda $, where $c$ is a bounded constant. \label{le:add:prelim:c}
\item 
Assumption \ref{A4} holds with $a=b=1$. \label{le:add:prelim:d}
\end{enumerate}

\end{Lemma}
Lemma \ref{le:add:prelim} (4) establishes a concentration inequality of $\xi_j$ for the additive model, such that $\max_{1\leq j \leq s} = O_P(\sqrt{\frac{\log N}{nh}})$. The proof is based on the extension of the Green function techniques (\cite{eggermont2009maximum}) to diverging dimensional setting; see Lemma \ref{le:add:EL09} in Appendix.

Combining Lemma \ref{le:add:prelim}, Theorems \ref{thm:of:thm1}, \ref{basic:thm:testing:null:distr} and \ref{thm:power:random:design}, we have the following result. 

\begin{Corollary}\label{cor:additive}
\begin{enumerate}
\item Suppose Assumptions \ref{A1}, \ref{A2} hold. If $m>1/2$, $d=o(N^{\frac{2m}{2m+1}}/\log N)$,  $s=o(d^{-1}N^{\frac{2m}{2m+1}}/\log N)$,  $\lambda \asymp N^{-\frac{2m}{2m+1}}$, then $\|\bar{f}-f_0\| = O_P(d^{1/2}N^{-\frac{m}{2m+1}})$. \label{cor:additive:est}

\item Suppose Assumptions \ref{A1}, \ref{A2} hold. If $m>3/4$, $d=o(N^{\frac{4m-3}{4(2m+1)}}(\log N)^{-\frac{4m+1}{4(2m+1)}})$,  $s=o(d^{-\frac{4(2m+1)}{4m+1}}N^{\frac{4m-3}{4m+1}}/\log N)$, and $\lambda \asymp d^{-\frac{2m}{4m+1}}N^{-\frac{4m}{4m+1}}$, then the Wald-type test achieves minimax rate of testing with $d^{\frac{2m+1}{2(4m+1)}}N^{-\frac{2m}{4m+1}}$. \label{cor:additive:test}
\end{enumerate}
\end{Corollary}

\begin{Remark}
It was shown by \cite{raskutti2012minimax} that $d^{1/2}N^{-\frac{m}{2m+1}}$ is the minimax estimation rate
in nonparametric additive model. 
Part (\ref{cor:additive:est}) of Corollary \ref{cor:additive}
provides an upper bound for $s$ such that $\bar{f}$ achieves this rate.
Meanwhile, Part (\ref{cor:additive:test}) of Corollary \ref{cor:additive} provides a different upper bound for $s$ such that our Wald-type test achieves
minimax rate of testing $d^{\frac{2m+1}{2(4m+1)}}N^{-\frac{2m}{4m+1}}$.
It should be emphasized that such minimax rate of testing is a new result in literature which is of independent interest.
The proof is based on a local geometry approach recently developed by \cite{wei2017local}.
When $d=1$, all results in this section reduce to Example 1 on univariate smoothing splines. 

\end{Remark}

\subsubsection{Example 3: Gaussian RKHS regression}\label{eg:example2}
Suppose that $\mathcal{H}$ is an RKHS generated by the Gaussian kernel $K(x,x')=\exp(-c\|x-x'\|^2),x,x'\in\bbR^d$,
where $c, d>0$ are constants. Here we consider $d=1,2$. Then Assumption \ref{A3} holds with
$\mu_\nu\asymp [(\sqrt{5}-1)/2]^{-(2\nu+1)}$, $\nu\ge1$;
see \cite{sollich2005understanding}.
It can be shown that $h\asymp (-\log{\lambda})^{-1/2}$ holds.
To verify Assumption \ref{A4}, we need the following lemma.
\begin{Lemma}\label{lemma:A3:verify}
For Gaussian RKHS,
Assumption \ref{A4} holds with $a=2$, $b=d+2$. 
\end{Lemma}

Following Theorem \ref{thm:of:thm1}, Theorems \ref{basic:thm:testing:null:distr} and \ref{thm:power:random:design},
we get the following consequence.
\begin{Corollary}\label{Gaussian:RKHS:cor}
Suppose that $\mathcal{H}$ is a Gaussian RKHS
and Assumptions \ref{A1} and \ref{A2} hold.
\begin{enumerate}
\item 
If $s=o(N/\log^{d+3}(N))$ and $\lambda\asymp N^{-1}\sqrt{\log{N}}$,
then $\|\bar{f}-f_0\|=O_P(N^{-1/2}\log^{1/4}{N})$.
\item 
If $s=o(N/\log^{d+3.5}{N})$ and $\lambda\asymp N^{-1}\log^{1/4}{N}$,
then the Wald-type test achieves minimax rate of testing $N^{-1/2}\log^{1/8}{N}$.
\end{enumerate}
\end{Corollary}
Corollary \ref{Gaussian:RKHS:cor} shows that 
one can choose $s$ to be of order $N$ (upto a logarithmic factor)
to obtain both optimal estimation and testing. 
This is consistent with the upper bound obtained by \cite{zhang2013divide} for optimal estimation, 
which is of a different logarithmic factor. Interestingly, Corollary \ref{Gaussian:RKHS:cor} shows that 
one can also choose $s$ to be almost identical to $N$ to obtain optimal testing.

\subsubsection{Example 4: Thin-Plate spline regression} \label{eg:example3}
Consider the $m$th order Sobolev space on $\mathbb{I}^d$, i.e.,
$\mathcal{H}= S^m(\mathbb{I}^d)$, with $d=2$ being fixed. 
It is known that 
Assumption \ref{A3} holds with $\mu_\nu \asymp \nu^{-2m/d}$; see \cite{gu2013smoothing}. Hence $h \asymp \lambda^{d/(2m)}$. The following lemma verifies Assumption \ref{A4}.
\begin{Lemma}\label{le:thin-plate}
 For thin-plate splines, Assumption \ref{A4} holds with $a=3-d/(2m)$, $b=1$.
\end{Lemma}
Following Theorem \ref{thm:of:thm1}, Theorem \ref{basic:thm:testing:null:distr} and Theorem \ref{thm:power:random:design}, we have the following result.
\begin{Corollary}\label{thin-plate:cor}
 Suppose $f \in S^m(\mathbb{I}^d)$ with $d=2$, Assumption \ref{A1} and Assumption \ref{A2} hold.
\begin{enumerate}
 \item  If $s=o(N^{\frac{(2m-d)^2}{2m(2m+d)}}/\log N)$ and $\lambda \asymp N^{-\frac{2m}{2m+d}}$, then $\|\bar{f}-f_0\| = O_P(N^{-m/(2m+d)})$. 

 \item If $s=o(N^\frac{4m^2-7dm+d^2}{(4m+d)m}/\log N)$ and $\lambda \asymp N^{-\frac{4m}{4m+d}}$, then the Wald-type test achieves minimax rate of testing $N^{-2m/(4m+d)}$.
\end{enumerate}
\end{Corollary}  
Corollary \ref{thin-plate:cor} demonstrates upper bounds on $s$.
These upper bounds are smaller compared with Corollary \ref{smoothing:spline:cor} in the univariate case, since the proof technique in bounding the empirical process $\xi_j$ here is not as sharp as the Green function technique used in Proposition \ref{prop:EL09} for the univariate example.

\section{Simulation}\label{sec:num:smooth}
In this section, we examined the performance of our proposed estimation and testing procedures versus various choices of number of machines in three examples based on simulated datasets. 

\subsection{Smoothing spline regression}
The data were generated from the following regression model 
\begin{equation}\label{model:eg1}
Y_i = c*(0.6\sin(1.5\pi X_i)) + \epsilon_i,\,\,\,\,i=1,\cdots,N, 
\end{equation}
where $X_i\overset{iid}{\sim}\text{Unif}[0,1]$, $\epsilon_i\overset{iid}{\sim} N(0,1)$ and $c$ is a constant. 
Cubic spline (i.e., $m=2$ in Section \ref{eg:example1}) was employed for estimating the regression function.
To display the impact of the number of divisions $s$ on statistical performance,
we set sample sizes $N=2^l$ for $9\le l\le 13$ and chose $s=N^\rho$
for $0.1\leq \rho \leq 0.8$. To examine the estimation procedure, 
we generated data from model (\ref{model:eg1}) with $c=1$. 
Mean squared errors (MSE) were reported based on 100 independent replicated experiments.
The left panel of Figure \ref{fig:smoothing} summarizes the results.
Specifically, it displays that the MSE increases
as $s$ does so; while the MSE increases suddenly when $\rho\approx0.7$, where
$\rho\equiv\log(s)/\log(N)$. 
Recall that the theoretical upper bound for $s$, is $N^{0.8}$;
see Corollary \ref{smoothing:spline:cor}.
Hence, estimation performance becomes worse near this theoretical boundary.

We next consider the hypothesis testing problem $H_0: f=0$. 
To examine the proposed Wald test,
we generated data from model (\ref{model:eg1}) at both $c=0,1$;
$c=0$ used for examining the size of the test,
and $c=1$ used for examining the power of the test.
Significance level was chosen as 0.05.
Both size and power were calculated as the proportions of rejections based on 500 independent replications. The middle and right panels of Figure \ref{fig:smoothing}
summarize the results.
Specifically, the right panel shows that the size approaches the nominal level $0.05$ under 
various choices of $(s,N)$, showing the validity of the Wald test.
The middle panel displays that
the power increases when $\rho$ decreases; the power maintains at $100\%$ when $\rho\le 0.5$ and $N\geq 4096$. Whereas the power quickly drops to zero when $\rho\ge0.6$. This is consistent
with our theoretical finding. Recall that the theoretical upper bound for $s$ is $N^{0.56}$; see Corollary \ref{smoothing:spline:cor}.
The numerical results also reveal that the upper bound of $s$ to achieve optimal testing is
indeed smaller than the one required for optimal estimation. 

\begin{figure}[h]\label{fig:smoothing}
  \centering
  \includegraphics[scale=0.35]{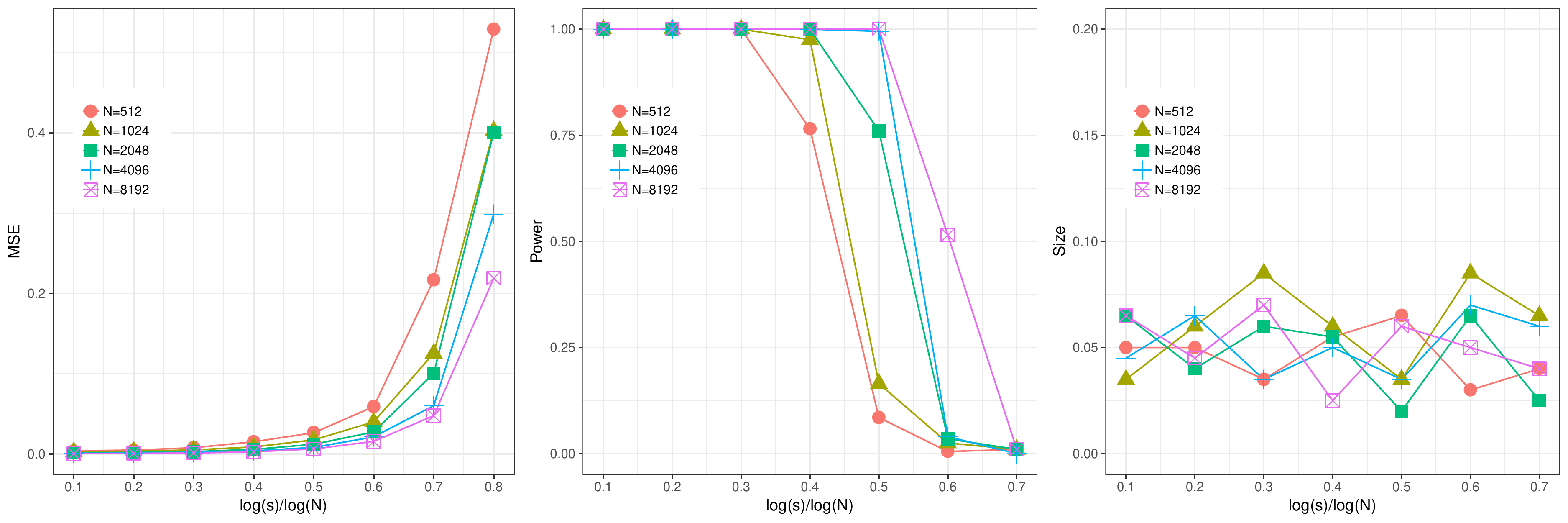}
  \caption{Smoothing Spline Regression. (a) MSE of $\bar{f}$ versus $\rho\equiv\log(s)/\log(N)$.
  (b) Power of the Wald test versus $\rho$.
  (c) Size of the Wald test versus $\rho$.}
\end{figure}

\subsection{Nonparametric additive regression}
We generated data from the following nonparametric model of two additive components
\begin{equation}\label{eq:add}
Y_i = c*f(X_{i1},X_{i2}) + \epsilon_i,\,\,i=1,\cdots,N,
\end{equation}
where $f(x_1,x_2) = 0.4\sin(1.5\pi x_1)+ 0.1(0.5-x_2)^3$, 
and $X_{i1},X_{i2}\overset{iid}{\sim}\text{Unif}[0,1]$, 
$\epsilon_i \overset{iid}{\sim} N(0,1)$, and $c$ is a constant. 
To examine the estimation procedure, 
we generated data from (\ref{eq:add}) with $c=1$.
To examine the testing procedure, we generated data at $c=0,1$.
$N,s$ were chosen to be the same as the smoothing spline example in Section \ref{sec:num:smooth}.
Results are summarized in Figure \ref{fig:additive}.
The interpretations are again similar to Figure \ref{fig:smoothing}, only with a slightly different asymptotic  trend. 
Specifically,
the MSE suddenly increases at $\rho \approx 0.6$, and the power quickly approaches one at $\rho \approx 0.5$.  The sizes are around the nominal level 0.05 for all cases.
\begin{figure}[h]\label{fig:additive}
  \centering
  \includegraphics[scale=0.35]{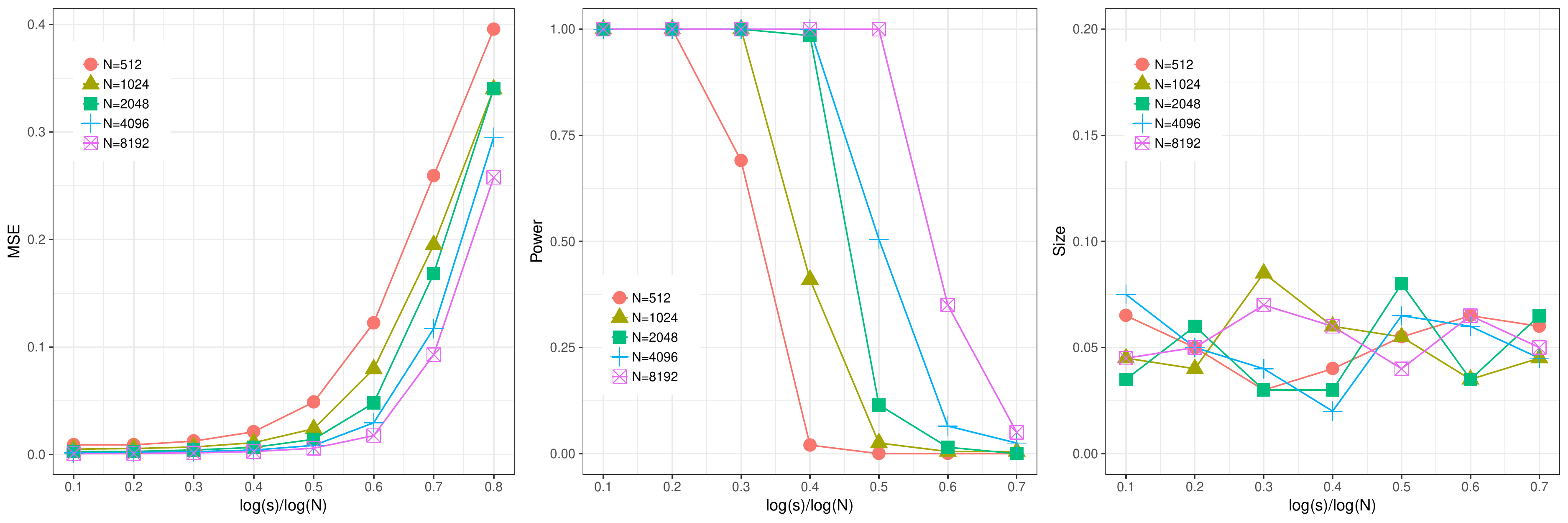}
  \caption{Additive Regression Model. (a) MSE of $\bar{f}$ versus $\rho\equiv\log(s)/\log(N)$.
  (b) Power of the Wald test versus $\rho$.
  (c) Size of the Wald test versus $\rho$.}
\end{figure}

\section{Conclusion}
Our work offers theoretical insights on how to allocate data in parallel computing for KRR in both estimation and testing procedures. In comparison with \cite{zhang2013divide} and \cite{shang2017computational}, our work provides a general and unified treatment of such problems in modern diverging-dimension or big data settings. Furthermore, using the green function for equivalent kernels to provide a sharp concentration bound on the empirical processes related to $s$, we have improved the upper bound of the number of machines in smoothing spline regression by \cite{zhang2013divide} from $N^{(2m-1)/(2m+1)}/\log{N}$ to $N^{2m/(2m+1)}/\log{N}$ for optimal estimation, which is proven un-improvable in \cite{shang2017computational} (upto a logarithmic factor). In the end, we would like to point out that our theory is useful in designing a distributed version of generalized cross validation method that is developed to choose tuning parameter $\lambda$ and the number of machines $s$; see \cite{xu2016optimal}.

\appendix
\section{Proofs of main results}
\subsection{Notation table}
\begin{table}[h]
    \centering
\begin{tabular}{c|p{0.6\textwidth}}
\hline
$N$ & sample size\\
  $Y$ & response \\
  $X$ & covariate \\
  $\epsilon$ & random error \\
  $\mathcal{H}$ & reproducing kernel Hilbert space (RKHS) \\
  $\pi(x)$ & density distribution\\
  $d$ & dimension of covariate \\
  $\langle \cdot, \cdot \rangle_{\mathcal{H}} $, $\|\cdot\|_{\mathcal{H}}$ & the inner product and norm under $\mathcal{H}$ \\
  $R(\cdot, \cdot)$ & kernel function under the norm $\|\cdot \|_{\mathcal{H}}$ \\
  $\mu_i$ & eigenvalue  \\
  $\varphi$ & eigenfunction \\
  $\langle \cdot, \cdot \rangle _{L^2_\pi (\mathcal{X})}$ & $L_2$ inner product  \\
  $\langle \cdot, \cdot \rangle $, $\|\cdot \|$ & embedded inner product and norm   \\
  $V(\cdot, \cdot)$ & $L_2$ inner product  \\
  $K(\cdot, \cdot)$ & kernel function equipped with $\|\cdot \|$  \\
  $K_x(\cdot)$ & $=K(x,\cdot)$\\
  $s$ & number of division \\
  $I_j$ & the set of indices of the observation from subsample $j$  \\
  $n$ & the subsample size  \\
  $\widehat{f}_j$ & the estimate of $f$ based on subsample $j$  \\
  $\lambda$ & penalization parameter\\
  $\bar{f}$ & D\&C estimator\\
  $T_{N,\lambda}$ & test statistic  \\
  $\|\cdot\|_{sup}$ & the supremum norm \\
  $h^{-1}$ & $= \sum_{\nu\geq 1} \frac{1}{1+\lambda/\mu_\nu}$  \\
  $\xi_j$ & $=\sup_{\substack{f,g\in\mathcal{H}\\ \|f\|=\|g\|=1}}|P_jfg-Pfg|$  \\
  $\mathcal{P}_\lambda$ &  self-adjoint operator satisfies $\langle \mathcal{P}_\lambda f,g\rangle=\lambda\langle f,g\rangle_{\mathcal{H}}$\\
  $\textbf{K}$ & empirical kernel matrix \\
  $S^m(\mathbb{I})$ & the $m$th order Sobolev space on $\mathbb{I}\equiv [0,1]$ \\
  \hline
\end{tabular}\\
\caption{A table that lists all useful notation and their meanings. }
\end{table}

\subsection{Some preliminary results}
\begin{Lemma}\label{le:prelim}
\begin{enumerate}
\item For any $x,y \in \mathcal{X}$, $K(x,y) \leq c_{\varphi}^2h^{-1}$.
\item For any $f\in \mathcal{H}$, $\|\mathcal{P}_\lambda f\| \leq \lambda^{1/2}\|f\|_{\mathcal{H}}$.
\end{enumerate}
\end{Lemma}

\begin{proof}
$(a)$ 
\begin{equation*}
K(x,y) = \sum_{\nu\geq 1} \frac{\varphi_\nu(x)\varphi_\nu(y)}{1+\lambda /\mu_\nu} \leq c_\varphi^2 h^{-1},
\end{equation*}
where the last inequality is by Assumption \ref{A3} and the definition of $h^{-1}$. \\
$(b)$ 
\begin{eqnarray*}
\|\mathcal{P}_\lambda f\| = &\sup_{g\in \mathcal{H}, \|g\|\leq 1} \langle \mathcal{P}_\lambda f, g \rangle = \sup_{g\in \mathcal{H}, \|g\|\leq 1} \lambda \langle f, g \rangle_{\mathcal{H}} 
\leq  \sup_{g\in \mathcal{H}, \|g\|\leq 1} \lambda^{1/2}\|f\|_{\mathcal{H}}\lambda^{1/2}\|g\|_{\mathcal{H}}
\leq \lambda^{1/2} \|f\|_{\mathcal{H}}.
\end{eqnarray*}
\end{proof}

\subsection{Proofs in Section \ref{sec:est}}

Our theoretical analysis relies on a set of Fr\'{e}chet derivatives to be specified below:
for $j=1,2,\ldots,s$,
the Fr\'{e}chet derivative of $\ell_{j,\lambda}$ can be identified as: for any $f,f_1,f_2\in\mathcal{H}$,
\begin{eqnarray*}
D\ell_{j,\lambda}(f)f_1&=&-\frac{1}{n}\sum_{i\in I_j}(Y_i-f(X_i))
\langle K_{X_i},f_1\rangle+\langle
\mathcal{P}_\lambda f,f_1\rangle:=\langle S_{j,\lambda}(f),f_1 \rangle,\\
DS_{j,\lambda}(f)f_1f_2&=&\frac{1}{n}\sum_{i\in I_j}f_2(X_i)\langle K_{X_i},f_1\rangle
+\langle\mathcal{P}_\lambda f_2,f_1\rangle=
\langle DS_{j,\lambda}(f)f_2,f_1\rangle,\\
D^2S_{j,\lambda}(f)&\equiv&0.
\end{eqnarray*}
More specifically,
\begin{eqnarray*}
S_{j,\lambda}(f)&=&-\frac{1}{n}\sum_{i\in I_j}(Y_i-f(X_i))K_{X_i}+\mathcal{P}_\lambda f,\\
DS_{j,\lambda}(f)g&=&\frac{1}{n}\sum_{i\in I_j}g(X_i)K_{X_i}+\mathcal{P}_\lambda g.
\end{eqnarray*}
Define $S_\lambda(f)=E\{S_{j,\lambda}(f)\}$, hence, $DS_\lambda(f)=E\{DS_{j,\lambda}(f)\}$.
It follows from \cite{shang2013local} that
$$\langle DS_\lambda(f)f_1,f_2\rangle=\langle f_1,f_2\rangle$$
for any $f,f_1,f_2\in\mathcal{H}$ which leads to $DS_\lambda(f)=id$.

\begin{proof}[Proof of Lemma \ref{basic:thm:estimation}]
Throughout the proof, let $\widetilde{f}_j=E\{\widehat{f}_j|\textbf{X}_j\}$.
It is easy to see that 
\begin{eqnarray*}
0&=&S_{j,\lambda}(\widehat{f}_j)=-\frac{1}{n}\sum_{i\in I_j}(Y_i-\widehat{f}_j(X_i))K_{X_i}+\mathcal{P}_\lambda\widehat{f}_j,\\
0&=&\frac{1}{n}\sum_{i\in I_j}(\widetilde{f}_j(X_i)-f_0(X_i))K_{X_i}+\mathcal{P}_\lambda\widetilde{f}_j.
\end{eqnarray*}
Subtracting the two equations one gets that
\begin{equation}\label{thm1:eqn:0}
\frac{1}{n}\sum_{i\in I_j}(\widehat{f}_j-\widetilde{f}_j)(X_i)K_{X_i}+\mathcal{P}_\lambda(\widehat{f}_j-\widetilde{f}_j)=\frac{1}{n}\sum_{i\in I_j}\epsilon_i K_{X_i}.
\end{equation}
Equation (\ref{thm1:eqn:0}) shows that 
\[
\widehat{f}_j-\widetilde{f}_j=\argmin_{f\in\mathcal{H}}\ell_{j,\lambda}^\star(f)
\equiv\argmin_{f\in\mathcal{H}}\frac{1}{2n}\sum_{i\in I_j}(\epsilon_i-f(X_i))^2
+\frac{\lambda}{2}\|f\|_{\mathcal{H}}^2.
\]

Let $e_j=\frac{1}{n}\sum_{i\in I_j}\epsilon_i K_{X_i}$ and $\varepsilon_j=\widehat{f}_j-\widetilde{f}_j$.
Then consider Taylor's expansion
\begin{eqnarray*}
\ell_{j,\lambda}^\star(e_j)-\ell_{j,\lambda}^\star(\varepsilon_j)&=&\frac{1}{2}D^2\ell_{j,\lambda}^\star(\varepsilon_j)(e_j-\varepsilon_j)(e_j-\varepsilon_j)\\
&=&\frac{1}{2}P_j(e_j-\varepsilon_j)^2+\frac{1}{2}\langle\mathcal{P}_\lambda (e_j-\varepsilon_j),e_j-\varepsilon_j\rangle,\\
\ell_{j,\lambda}^\star(\varepsilon_j)-\ell_{j,\lambda}^\star(e_j)&=&
D\ell_{j,\lambda}^\star(e_j)(\varepsilon_j-e_j)+\frac{1}{2}D^2\ell_{j,\lambda}^\star(e_j)
(\varepsilon_j-e_j)(\varepsilon_j-e_j)\\
&=&(P_j-P)(e_j(\varepsilon_j-e_j))+\frac{1}{2}P_j(\varepsilon_j-e_j)^2+
\frac{1}{2}\langle\mathcal{P}_\lambda(\varepsilon_j-e_j),\varepsilon_j-e_j\rangle.
\end{eqnarray*}
Adding the two equations one obtains that
\begin{eqnarray*}
&&P_j(\varepsilon_j-e_j)^2+
\langle\mathcal{P}_\lambda(\varepsilon_j-e_j),\varepsilon_j-e_j\rangle+(P_j-P)(e_j(\varepsilon_j-e_j))=0.
\end{eqnarray*}

Uniformly for $j$, it holds that 
\begin{eqnarray*}
|(P_j-P)(e_j(\varepsilon_j-e_j))|&\le&\xi_j\|e_j\|\cdot\|\varepsilon_j-e_j\|,\\
P_j(\varepsilon_j-e_j)^2+\langle\mathcal{P}_\lambda(\varepsilon_j-e_j),(\varepsilon_j-e_j)\rangle
&\ge&(1-\xi_j)\|\varepsilon_j-e_j\|^2.
\end{eqnarray*}
Combining the two inequalities one gets that
\[
(1-\xi_j)\|\varepsilon_j-e_j\|^2\le\xi_j\|e_j\|\cdot\|\varepsilon_j-e_j\|.
\]
Taking expectations conditional on $\textbf{X}_j$ on both sides and noting that $\xi_j$ is $\sigma(\textbf{X}_j)$-measurable,
one gets that
\[
(1-\xi_j)E\{\|\varepsilon_j-e_j\|^2|\textbf{X}_j\}\le\xi_jE\{\|e_j\|\cdot\|\varepsilon_j-e_j\||\textbf{X}_j\}\le\xi_j E\{\|e_j\|^2|\textbf{X}_j\}^{1/2}E\{\|\varepsilon_j-e_j\|^2|\textbf{X}_j\}^{1/2}.
\]
By assumption
$\log^b{N}=o(nh^a)$ and Assumption \ref{A4}, 
$\max_{1\le j\le s}\xi_j=o_P(1)$, i.e., with probability approaching one 
$\max_{1\le j\le s}\xi_j\le 1/2$, hence, 
\begin{eqnarray}\label{thm1:eqn:1}
E\{\|\varepsilon_j-e_j\|^2|\textbf{X}_j\}&\le& 4\xi_j^2 E\{\|e_j\|^2|\textbf{X}_j\}\nonumber\\
&=&\frac{4\xi_j^2}{n^2}\sum_{i,i'\in I_j}E\{\epsilon_i\epsilon_{i'}K(X_i,X_{i'})|\textbf{X}_j\}\nonumber\\
&=&\frac{4\xi_j^2}{n^2}\sum_{i\in I_j}\sigma^2(X_i)K(X_i,X_i)\nonumber\\
&\le&\frac{4 c_\pi c_\varphi^2\xi_j^2}{nh},
\end{eqnarray}
where the last inequality follows from Assumption \ref{A1} and Lemma \ref{le:prelim} that $K(x,x)\le c_\varphi^2h^{-1}$.
This proves (\ref{basic:thm:estimation:conc:-1}).

By (\ref{thm1:eqn:1}) it is easy to derive
\begin{equation}\label{thm1:eqn:2}
E\{\|\widehat{f}_j-\widetilde{f}_j\|^2|\textbf{X}_j\}
\le \frac{4 c_\pi c_\varphi^2}{nh}.
\end{equation}

Now we look at $\|\widetilde{f}_j-f_0^\star\|$, where $f_0^\star=(id-\mathcal{P}_\lambda)f_0$.  
It is easy to see that $\widetilde{f}_j$ is the minimizer of the following problem.
\[
\widetilde{f}_j=\argmin_{f\in\mathcal{H}}\widetilde{\ell}_{j,\lambda}(f)\equiv
\argmin_{f\in\mathcal{H}}\frac{1}{2n}\sum_{i\in I_j}(f_0(X_i)-f(X_i))^2+\frac{\lambda}{2}\|f\|_{\mathcal{H}}^2.
\]
We use a similar strategy for handling part (\ref{basic:thm:estimation:conc:-1}).
Note that
\begin{eqnarray*}
\widetilde{\ell}_{j,\lambda}(f_0^\star)-\widetilde{\ell}_{j,\lambda}(\widetilde{f}_j)
&=&\frac{1}{2}D^2\widetilde{\ell}_{j,\lambda}(\widetilde{f}_j)(f_0^\star-\widetilde{f}_j)
(f_0^\star-\widetilde{f}_j)\\
&=&\frac{1}{2}P_j(f_0^\star-\widetilde{f}_j)^2+\frac{1}{2}\langle\mathcal{P}_\lambda(f_0^\star-\widetilde{f}_j),
f_0^\star-\widetilde{f}_j\rangle,\\
\widetilde{\ell}_{j,\lambda}(\widetilde{f}_j)-\widetilde{\ell}_{j,\lambda}(f_0^\star)
&=&P_j(f_0^\star-f_0)(\widetilde{f}_j-f_0^\star)+\langle\mathcal{P}_\lambda f_0^\star,\widetilde{f}_j-f_0^\star\rangle\\
&&+\frac{1}{2}P_j(\widetilde{f}_j-f_0^\star)^2+\frac{1}{2}\langle\mathcal{P}_\lambda(\widetilde{f}_j-f_0^\star),
\widetilde{f}_j-f_0^\star\rangle.
\end{eqnarray*}
Adding the two equations, one gets that
\begin{eqnarray*}
&&P_j(\widetilde{f}_j-f_0^\star)^2+\langle\mathcal{P}_\lambda(\widetilde{f}_j-f_0^\star),
\widetilde{f}_j-f_0^\star\rangle\\
&=&P_j(f_0-f_0^\star)
(\widetilde{f}_j-f_0^\star)-\langle\mathcal{P}_\lambda f_0^\star,
\widetilde{f}_j-f_0^\star\rangle\\
&=&(P_j-P)(f_0-f_0^\star)(\widetilde{f}_j-f_0^\star)+P(f_0-f_0^\star)(\widetilde{f}_j-f_0^\star)-\langle
\mathcal{P}_\lambda f_0^\star,\widetilde{f}_j-f_0^\star\rangle\\
&=&(P_j-P)(f_0-f_0^\star)(\widetilde{f}_j-f_0^\star)+\langle f_0-f_0^\star,\widetilde{f}_j-f_0^\star\rangle \\
& &-\langle\mathcal{P}_\lambda(f_0-f_0^\star),\widetilde{f}_j-f_0^\star\rangle-\langle
\mathcal{P}_\lambda f_0^\star,\widetilde{f}_j-f_0^\star\rangle\\
&=&(P_j-P)(f_0-f_0^\star)(\widetilde{f}_j-f_0^\star)+\langle
f_0-f_0^\star-\mathcal{P}_\lambda(f_0-f_0^\star)-\mathcal{P}_\lambda f_0^\star,\widetilde{f}_j-f_0^\star\rangle\\
&=&(P_j-P)(f_0-f_0^\star)(\widetilde{f}_j-f_0^\star).
\end{eqnarray*}
Therefore, 
\[
(1-\xi_j)\|\widetilde{f}_j-f_0^\star\|^2\le\xi_j\|f_0-f_0^\star\|\times\|\widetilde{f}_j-f_0^\star\|
=\xi_j\|\mathcal{P}_\lambda f_0\|\times \|\widetilde{f}_j-f_0^\star\|
\le C\xi_j\lambda^{1/2}\|f_0\|_{\mathcal{H}}\|\widetilde{f}_j-f_0^\star\|,
\]
implying that, with probability approaching one, for any $1\le j\le s$,
$\|\widetilde{f}_j-f_0^\star\|\le 2C\xi_j\lambda^{1/2}\|f_0\|_{\mathcal{H}}$.
This proves (\ref{basic:thm:estimation:conc:0}).
\end{proof}

\begin{proof}[Proof of Theorem \ref{thm:of:thm1}]
Recall $f_0^\star=(id-\mathcal{P}_\lambda)f_0$ and $\widetilde{f}_j=E\{\widehat{f}_j|\textbf{X}_j\}$.
Also notice that $\frac{1}{N}\sum_{i=1}^N\epsilon_i K_{X_i}=\frac{1}{s}\sum_{j=1}^s e_j$.
By direct calculations and Lemma \ref{basic:thm:estimation}, we have with probability approaching one,
\begin{eqnarray*}
&&
E\{\|\bar{f}-f_0^\star-\frac{1}{N}\sum_{i=1}^N\epsilon_i K_{X_i}\|^2|\textbf{X}\}\\
&=&\frac{1}{s^2}\sum_{j=1}^s E\{\|\widehat{f}_j-\widetilde{f}_j-e_j\|^2|\textbf{X}_j\}
+\frac{1}{s^2}\|\sum_{j=1}^s(\widetilde{f}_j-f_0^\star)\|^2\\
&\le&4\left(\frac{ c_\pi  c_\varphi^2}{Nh}+\lambda \|f_0\|_{\mathcal{H}}^2\right)\max_{1\le j\le s}\xi_j^2.
\end{eqnarray*}
This proves (\ref{cor:thm1:eqn:0}). The result (\ref{cor:thm1:eqn:1}) immediately follows
by the assumption $\max_{1\le j\le s}\xi_j^2=o_P(1)$.
\end{proof}

\subsection{Proofs in Section \ref{sec:general:testing}}

\begin{proof}[Proof of Lemma \ref{basic:lemma:for:testing}]
It is easy to see that 
\[
\boldsymbol{\epsilon}'\textbf{K}\boldsymbol{\epsilon}
=\sum_{i=1}^N\epsilon_i^2 K(X_i,X_i)+W(N).
\]
Since 
\[
Var\left(\sum_{i=1}^N\epsilon_i^2 K(X_i,X_i)\right)\le NE\{\epsilon_i^4 K(X_i,X_i)^2\}
\le \tau c_\varphi^4Nh^{-2},
\]
where the last ``$\le$" follows by Assumption \ref{A2} and Lemma \ref{le:prelim} that $K(x,x)\le c_\varphi^2 h^{-1}$,
we get that
\begin{eqnarray*}
\sum_{i=1}^N\epsilon_i^2 K(X_i,X_i)&=&E\{\sum_{i=1}^N\epsilon_i^2 K(X_i,X_i)\}
+O_P\left(\sqrt{c_\varphi^4 Nh^{-2}}\right)\\
&=&\sigma^2Nh^{-1}+O_P(\sqrt{c_\varphi^4Nh^{-2}}).
\end{eqnarray*}

Next we prove asymptotic normality of $W(N)$.
Note $\sigma^2(N)=E\{W(N)^2\}$. Let
$G_I$, $G_{II}$, $G_{IV}$ be defined as
\begin{eqnarray*}
G_I&=&\sum_{1\le i<t\le n} E\{W_{it}^4\},\\
G_{II}&=&\sum_{1\le i<t<k\le n} (E\{W_{it}^2W_{ik}^2\}+
E\{W_{ti}^2W_{tk}^2\}+
E\{W_{ki}^2W_{kt}^2\})\\
G_{IV}&=&\sum_{1\le i<t<k<l\le n} (E\{W_{it} W_{ik} W_{lt} W_{lk}\}+
E\{W_{it} W_{il} W_{kt} W_{kl}\}+ E\{W_{ik}W_{il}W_{tk}W_{tl}\}).
\end{eqnarray*}
Since $K(x,x)\le c_\varphi^2 h^{-1}$, we have
$G_I=O(N^2h^{-4})$ and $G_{II}=O(N^3  h^{-4})$.
It can also be shown that for pairwise distinct $i,k,t,l$,
\begin{eqnarray*}
&&E\{W_{ik}W_{il}W_{tk}W_{tl}\}\\
&=&2^4E\{\epsilon_i^2\epsilon_k^2\epsilon_t^2\epsilon_l^2
K(X_i,X_k)K(X_i,X_l)K(X_t,X_k)
K(X_t,X_l)\}\\
&=&2^4\sigma^8\sum_{\nu=1}^\infty\frac{1}{(1+\lambda/\mu_\nu)^4}=O(h^{-1}),
\end{eqnarray*}
which implies that $G_{IV}=O(N^4h^{-1})$.
In the mean time, a straight algebra leads to that
\begin{eqnarray*}
\sigma^2(N)&=&4\sigma^4{N \choose 2}\sum_{\nu=1}^\infty\frac{1}{(1+\lambda/\mu_\nu)^2}\\
&=&2\sigma^4 N(N-1)\sum_{\nu\ge 1}\frac{1}{(1+\lambda/\mu_\nu)^2}\asymp N^2 h^{-1},
\end{eqnarray*}
where the last conclusion follows by Proposition \ref{prop:var_eff}.
Thanks to the conditions $h\to0$,  
$Nh^2\to\infty$, 
$G_I,G_{II}$ and $G_{IV}$ are all of order $o(\sigma^4(N))$.
Then it follows by \cite{de1987central} that as $N\rightarrow\infty$,
\[
\frac{W(N)}{\sigma(N)}\overset{d}{\longrightarrow}N(0,1).
\]
The above limit leads to that $W(N)=O_P(Nh^{-1/2})$.
\end{proof}

\begin{proof}[Proof of Theorem \ref{basic:thm:testing:null:distr}]
The proof is based on Lemma \ref{basic:lemma:for:testing}.
Under $f_0=0$, it follows from Corollary \ref{thm:of:thm1} and Assumption \ref{A4} that
\[
E\{\|\bar{f}-\frac{1}{N}\sum_{i=1}^N\epsilon_i K_{X_i}\|^2|\textbf{X}\}=O_P\left(\frac{c_\varphi^2\log^b{N}}{Nnh^{1+a}}\right),
\]
leading to 
\[
\|\bar{f}-\frac{1}{N}\sum_{i=1}^N\epsilon_i K_{X_i}\|^2=O_P\left(\frac{c_\varphi^2\log^b{N}}{Nnh^{1+a}}\right).
\]
Following the proof of Lemma \ref{basic:thm:estimation} and the trivial fact $\widehat{f}_j=0$ when $f_0=0$,
we have for any $1\le j\le s$,
\begin{equation}\label{proof:bound}
E\{\|\widehat{f}_j-e_j\|^2|\textbf{X}_j\}\le \frac{4 c_\pi c_\varphi^2\xi_j^2}{nh},\,\,\,\,
E\{\|e_j\|^2|\textbf{X}_j\}\le\frac{ c_\pi c_\varphi^2}{nh},\,\,\,\,\textrm{a.s.}
\end{equation}
Therefore, by Cauchy-Schwartz inequality,
\[
E\{|\langle\widehat{f}_j-e_j,e_j\rangle|\big|\textbf{X}_j\}
\le \sqrt{E\{\|\widehat{f}_j-e_j\|^2|\textbf{X}_j\}E\{\|e_j\|^2|\textbf{X}_j\}}
\le \frac{2 c_\pi  c_\varphi^2}{nh}\xi_j,
\]
and hence, 
\begin{eqnarray*}
E\left\{\sum_{j=1}^s|\langle\widehat{f}_j-e_j,e_j\rangle|\bigg|\textbf{X}\right\}
\le\frac{2 c_\pi sc_\varphi^2}{nh}\max_{1\le j\le s}\xi_j.
\end{eqnarray*}
By Assumption \ref{A4}, the above leads to that
\[
\sum_{j=1}^s\langle\widehat{f}_j-e_j,e_j\rangle=O_P\left(\frac{sc_\varphi^2}{nh}\sqrt{\frac{\log^b{N}}{nh^a}}\right).
\]
Meanwhile, it holds that
\[
\sum_{j\neq l}\langle\widehat{f}_j-e_j,e_l\rangle=
\sum_{j<l}\langle\widehat{f}_j-e_j,e_l\rangle+
\sum_{j>l}\langle\widehat{f}_j-e_j,e_l\rangle
\equiv R_1+R_2,
\]
with 
\[
R_1=O_P\left(\frac{sc_\varphi^2}{nh}\sqrt{\frac{\log^b{N}}{nh^a}}\right),\,\,\,\,
R_2=O_P\left(\frac{sc_\varphi^2}{nh}\sqrt{\frac{\log^b{N}}{nh^a}}\right).
\]
To see this, note that
\begin{eqnarray*}
E\{R_1^2|\textbf{X}\}&=&\sum_{j<l}E\{|\langle\widehat{f}_j-e_j,e_l\rangle|^2|\textbf{X}\}\\
&\le&\sum_{j<l}E\{\|\widehat{f}_j-e_j\|^2\|e_l\|^2|\textbf{X}\}\\
&=&\sum_{j<l}E\{\|\widehat{f}_j-e_j\|^2|\textbf{X}_j\}E\{\|e_l\|^2|\textbf{X}_l\}\\
&\le&{s\choose 2}\frac{4 c_\pi ^2c_\varphi^4}{n^2h^2}\max_{1\le j\le s}\xi_j^2,
\end{eqnarray*}
where the last inequality is based on (\ref{proof:bound}).
Similar result holds for $R_2$.
Hence, by Lemma \ref{basic:lemma:for:testing} and direct algebra, we get that
\begin{eqnarray*}
T_{N,\lambda}&=&N^{-2}\boldsymbol{\epsilon}'\textbf{K}\boldsymbol{\epsilon}
+\frac{2}{s^2}\sum_{j,l=1}^s\langle\widehat{f}_j-e_j,e_l\rangle+\|\bar{f}-\frac{1}{N}\sum_{i=1}^N\epsilon_i
K_{X_i}\|^2\\
&=&N^{-2}\boldsymbol{\epsilon}'\textbf{K}\boldsymbol{\epsilon}
+\frac{2}{s^2}\sum_{j=1}^s\langle\widehat{f}_j-e_j,e_j\rangle
+\frac{2}{s^2}(R_1+R_2)+\|\bar{f}-\frac{1}{N}\sum_{i=1}^N\epsilon_i
K_{X_i}\|^2\\
&=&\frac{\sigma^2}{Nh}+\frac{W(N)}{N^2}+O_P\left(\frac{c_\varphi^2}{N^{3/2}h}\right)+O_P\left(
\frac{c_\varphi^2}{Nh}\sqrt{\frac{\log^b{N}}{nh^a}}\right)+O_P\left(\frac{c_\varphi^2\log^b{N}}{Nnh^{1+a}}\right)\\
&=&\frac{\sigma^2}{Nh}+\frac{W(N)}{N^2}+O_P\left(\frac{c_\varphi^2}{N^{3/2}h}\right)+O_P\left(
\frac{c_\varphi^2}{Nh}\sqrt{\frac{\log^b{N}}{nh^a}}\right).
\end{eqnarray*}
The last equality follows from the condition $\log^b{N}=o(nh^{a+1})$. 
Therefore, by $c_\varphi^4/(Nh)=o(1)$, 
$Nh\to\infty$ (from $Nh^2\to\infty$ and $h\to0$), condition
$\log^b{N}=o(nh^{a+1})$ 
and $\sigma^2(N)\asymp N^2h^{-1}$ (Lemma \ref{basic:lemma:for:testing}), 
as $N\to\infty$,
\begin{eqnarray*}
\frac{N^2}{\sigma(N)}\left(T_{N,\lambda}-\frac{\sigma^2}{Nh}\right)
&=&\frac{W(N)}{\sigma(N)}+O_P\left(\frac{c_\varphi^2}{\sqrt{Nh}}+c_\varphi^2\sqrt{\frac{\log^b{N}}{nh^{a+1}}}\right)\\
&=&\frac{W(N)}{\sigma(N)}+o_P(1)
\overset{d}{\to}N(0,1).
\end{eqnarray*}
Proof is completed.
\end{proof}

\begin{proof}[Proof of Theorem \ref{thm:power:random:design}]
For any $f\in\mathcal{H}$, 
define $R_f=\bar{f}-N^{-1}\sum_{i=1}^N\epsilon_i K_{X_i}-f+\mathcal{P}_\lambda f$.
By direct examinations, it holds that
\begin{eqnarray*}
&&\|\bar{f}\|^2-\sigma^2/(Nh)\\
&=&\|R_f+\frac{1}{N}\sum_{i=1}^N\epsilon_i K_{X_i}+f-\mathcal{P}_\lambda f\|^2-\sigma^2/(Nh)\\
&\ge&\left\{\boldsymbol{\epsilon}'\textbf{K}\boldsymbol{\epsilon}/N^2-\sigma^2/(Nh)\right\}
+\|f-\mathcal{P}_\lambda f\|^2-\frac{2}{N}\sum_{i=1}^N\epsilon_i(f-\mathcal{P}_\lambda f)(X_i)\\
&&+\frac{2}{N}\sum_{i=1}^N\epsilon_i R_f(X_i)-2\langle f-\mathcal{P}_\lambda f, R_f\rangle\\
&\equiv& T_1+T_2+T_3+T_4+T_5.
\end{eqnarray*}
It follows by (\ref{basic:lemma:for:testing:eqn:0}),
Theorem \ref{thm:of:thm1}, Assumption \ref{A4} that,
uniformly for $f\in\mathcal{H}$,
\begin{eqnarray*}
&& T_1=W(N)/N^2+O_P((N^{3/2}h)^{-1}),\,\,\,\,\textrm{(by (\ref{basic:lemma:for:testing:eqn:0}))}\\
&& P_f\left(|T_3|\ge \sigma\|f-\mathcal{P}_\lambda f\|/(\varepsilon\sqrt{N})\right)\le\varepsilon^2,
\,\,\,\,\textrm{for arbitrary $\varepsilon>0$}\\
&& T_4=O_P( b_{N,\lambda}/\sqrt{Nh}),\,\,\,\,\textrm{(by Theorem \ref{thm:of:thm1}, Assumption \ref{A4} and (\ref{basic:lemma:for:testing:eqn:0}))}\\
&& T_5=\|f-\mathcal{P}_\lambda f\|\times O_P(b_{N,\lambda}),
\,\,\,\,\textrm{(by Theorem \ref{thm:of:thm1} and Assumption \ref{A4})}
\end{eqnarray*}

Note that $\|\mathcal{P}_\lambda f\|\le \lambda^{1/2}\|f\|_{\mathcal{H}}$ for any $f\in \mathcal{H}$. 
Therefore, to achieve high power, i.e., power is at least $1-\varepsilon$, 
one needs to choose a large $N_\varepsilon$ and $C_\varepsilon$ s.t. $N\ge N_\varepsilon$ and
\begin{align*}
&\|f\|\ge C_\varepsilon/\sqrt{Nh^{1/2}},\,\,\,\,
\|f\|\ge C_\varepsilon/\sqrt{N},\,\,\,\,
\|f\|\ge C_\varepsilon \sqrt{ b_{N,\lambda}/\sqrt{Nh}},\,\,\,\, \\
&\|f\|\ge C_\varepsilon b_{N,\lambda},\,\,\,\,
\|f\|\ge C_\varepsilon \lambda^{1/2}\|f\|_{\mathcal{H}}. 
\end{align*}
Proof is completed.
\end{proof}

\subsection{Proofs in Section \ref{eg:example4}}
\begin{proof}[Proof of Lemma \ref{le:add:prelim} $(\ref{le:add:prelim:a})$]

 For each $\nu\geq 1$, there exist  $p\in \mathbb{N}$ and $1\leq k \leq d$, such that $\nu=pd+k$. Suppose $x=(x_1, \cdots, x_d)$, then for each $x_k$, there exists $(\varphi_p^{(k)}, \mu_p^{(k)})$ and $(\varphi_{p'}^{(k)}, \mu_{p'}^{(k)})$ satisfying 
$V_k(\varphi_p^{(k)},\varphi_{p'}^{(k)})=\delta_{pp'}$ and $\int_{\mathbb{I}}\varphi_p^{(k)}(x)\varphi_{p'}^{(k)}(x)dx=\delta_{pp'}/\mu_p^{(k)}$.
 In fact, the eigenfunctions $\varphi_\nu$ and eigenvalues $\mu_\nu$ can be constructed by an ordered sequence of $\varphi_p^{(k)}, \mu_p^{(k)}$ as  $\varphi_\nu (x) = \varphi_p^{(k)} (x_k)$ and $\mu_\nu= \mu_p^{(k)}$. 

Next, we verify such construction of eigenfunctions $\varphi_\nu$ and eigenvalues $\mu_\nu$ satisfy Assumption \ref{A3}. When $\nu \neq \mu$, then there exist $p_1, q_1, p_2, q_2$, such that $\nu= p_1d + q_1, \mu= p_2d+ q_2$, then 
\begin{align*}\label{eq:const:test}
V(\varphi_{p_1d+q_1}, \varphi_{p_2d+q_2}) & = V(\varphi_{p_1}^{q_1}(x_{q_1}), \varphi_{p_2}^{q_2}(x_{q_2}))\\
 & =
  \begin{cases}
   0      & \quad p_1 \neq p_2, q_1 = q_2\\
    V_{q_1}(\varphi_{p_1}^{q_1}(x_{q_1}),0) + V_{q_2}(0, \varphi_{p_2}^{q_2}(x_{q_2})) =0  & \quad q_1 \neq q_2\\
  \end{cases}
\end{align*}
On the other hand, 
\begin{equation*}
\langle \varphi_\nu, \varphi_\mu \rangle_{\mathcal{H}} = \langle \varphi_{p_1}^{q_1}, \varphi_{p_2}^{q_2}\rangle_{\mathcal{H}} = 
\begin{cases}
1/\mu_{p_1}^{q_1}= 1/\mu_\nu & \quad p_1 = p_2, q_1 = q_2\\
0 & \quad \nu \neq \mu\\
\end{cases}
\end{equation*}
For any $f\in \mathcal{H}$, 
\begin{align*}
f(x_1,\cdots, x_d) & = f_1(x_1) + \cdots + f_d(x_d)
 = \sum_{k=1}^d \sum_{\nu=1}^\infty V_k(f_k, \varphi_\nu^{(k)}) \varphi_\nu^{(k)}(x_k)\\
& = \sum_{k=1}^d \sum_{\nu=1}^\infty V(f, \varphi_\nu^{(k)})\varphi_\nu^{(k)}(x_k) = \sum_{\nu=1}^\infty V(f, \varphi_\nu)\varphi_\nu(x)
\end{align*}

\end{proof}

\begin{proof}[Proof of Lemma \ref{le:add:prelim} $(\ref{le:add:prelim:b})$]
It is easy to see that 
$$
\sum_{\nu\geq 1} (1+ \lambda/\mu_\nu)^{-1} = \sum_{q=1}^d \sum_{p\geq 1} (1+ \lambda/\mu_{p}^{(k)})^{-1} \asymp d\lambda^{-1/(2m)} := h^{-1}. 
$$
\end{proof}

\begin{proof}[Proof of Lemma \ref{le:add:prelim} $(\ref{le:add:prelim:c})$]
Notice that $\|f\|^2_{\mathcal{H}} \leq \sum_{i=1}^d \|f_k\|^2_{\mathcal{H}} \leq C d$, then by Lemma \ref{le:prelim} $(b)$, $\|\mathcal{P}_\lambda f\|^2 \leq \lambda \|f\|^2_{\mathcal{H}} \leq C d\lambda$. 
\end{proof}

Next, we prove Lemma \ref{le:add:prelim} (\ref{le:add:prelim:d}). To prove Lemma \ref{le:add:prelim} (\ref{le:add:prelim:d}), it is sufficient to prove the following Lemma \ref{le:add:EL09}.

 \begin{Lemma}\label{le:add:EL09}
Under (\ref{quasi:uniform:assumption}),
there exist universal positive constants $c_1,c_2,c_3$ such that for any $1\le j\le s$,
\[
P\left(\xi_j\ge t\right)\le 2n\exp\left(-\frac{nht^2}{c_1+c_2t}\right),\,\,
\textrm{for all $t\ge c_3(nh)^{-1}$}, 
\]
where $h^{-1} \asymp d \lambda^{-1/(2m)}$.
\end{Lemma}
The proof of Lemma \ref{le:add:prelim} is based on the green function for equivalent kernel technique in \cite{eggermont2009maximum}, see Supplement for details. 

\subsection{Proofs in Section \ref{sec:two:examples}}
\begin{proof}[Proof of Lemma \ref{lemma:A3:verify}]
For $p,\delta>0$,
define $\mathcal{G}(p)=\{f\in\mathcal{H}: \|f\|_{\sup}\le1,\|f\|_{\mathcal{H}}\le p\}$
and the corresponding entropy integral
\begin{equation}\label{eq:entropy}
J(p,\delta)=\int_0^\delta\psi_2^{-1}\left(D(\varepsilon,\mathcal{G}(p),\|\cdot\|_{\sup})\right)d\varepsilon
+\delta\psi_2^{-1}\left(D(\delta,\mathcal{G}(p),\|\cdot\|_{\sup})^2\right),
\end{equation}
where $\psi_2(s)=\exp(s^2)-1$ and $D(\varepsilon,\mathcal{G}(p),\|\cdot\|_{\sup})$
is the $\varepsilon$-packing number of $\mathcal{G}(p)$
in terms of $\|\cdot\|_{\sup}$-metric. In what follows, we particularly choose $p=c_K^{-1}(h/\lambda)^{1/2}$,
where $c_K\equiv\sup_{g\in\mathcal{H}}h^{1/2}\|g\|_{\sup}/\|g\|$ is finite,
according to \cite{yang2017non}. 

Define $\psi_i(g)=c_k^{-1}h^{1/2}g(X_i)$ and 
$Z_j(g)=n^{-1/2}\sum_{i\in I_j}[\psi_i(g)K_{X_i}-E\{\psi_i(g)K_{X_i}\}]$.
Following \cite[Lemma 6.1]{yang2017non},
for any $1\le j\le s$, for any $t\ge0$,
\begin{equation}\label{lemma:YSC17}
P\left(\sup_{g\in\mathcal{G}(p)}\|Z_j(g)\|\ge t\right)\le 2\exp\left(-\frac{t^2}{C^2 J(p,1)^2}\right),
\end{equation}
for an absolute constant $C>0$.
Since $\|f\|=1$ implies that $c_K^{-1}h^{1/2}f\in\mathcal{G}(p)$.
Then it can be shown that
\[
\sqrt{n}\xi_j\le c_K^2h^{-1}\sup_{g\in\mathcal{G}(p)}\|Z_j(g)\|,\,\,\,\,j=1,\ldots,s.
\]
Following (\ref{lemma:YSC17}) we have
\[
P\left(\sqrt{n}\max_{1\le j\le s}\xi_j\ge t\right)\le 2s\exp\left(-\frac{c_K^{-4}h^2t^2}{C^2J(p,1)^2}\right),
\]
which implies that 
\begin{equation}\label{eq:emprocess}
\sqrt{n}\max_{1\le j\le s}\xi_j=O_P\left(\sqrt{\frac{\log{N}}{h^2}}J(p,1)\right).
\end{equation}
It follows by \cite[Proposition 1]{zhou2002covering} that $J(p,1)=O\left([\log(h/\lambda)]^{(d+1)/2}\right)
=O\left([\log{N}]^{(d+1)/2}\right)$.
Then 
\[
\max_{1\le j\le s}\xi_j=O_P\left(\sqrt{\frac{\log^{d+2}{N}}{nh^2}}\right).
\]
That is, Assumption \ref{A4} holds with $a=2$ and $b=d+2$. Proof completed.
\end{proof}

\begin{proof}[Proof of Lemma \ref{le:thin-plate}]
\begin{align*}
J(p,1) & \leq \int_0^1 \sqrt{\log  D(\varepsilon,\mathcal{G}, \|\cdot\|_{\sup})}\,d\varepsilon + \sqrt{\log D(1,\mathcal{G}, \|\cdot\|_{\sup})}\\
& \leq \int_0^1 \sqrt{\left(\frac{p}{\varepsilon}\right)^{\frac{d}{m}}+1} \, d\varepsilon + \sqrt{2}p^{\frac{d}{2m}}\\
& \leq c'_d \, p^{d/(2m)}
\end{align*}
where the penultimate step is based on \cite{poggio2002mathematical}. Therefore, $J(p,1) = O(p^{\frac{d}{2m}})$, where $p=(h/\lambda)^{1/2}$. From e.q.(\ref{eq:emprocess}), we have 
$$
\max_{1\leq j \leq s} \xi_j  =O_P\left(\sqrt{\frac{\log N}{nh^{3-d/(2m)}}}\right)
$$
\end{proof}

\section{Some technical proofs}
\subsection{Proof of Proposition 3.1}\label{proof:prop_var}
\begin{proof}
Define 
$$
s_\lambda = \argmin\{j: \mu_j \leq \lambda\} -1,
$$
that is, $s_\lambda$ is the number of eigenvalues that are greater than $\lambda$. Then the effective dimension can be written as 
$$
h^{-1} = \sum_{j=1}^{\infty} \frac{\mu_j}{\mu_j+ \lambda} = \sum_{j=1}^{s_\lambda} \frac{\mu_j}{\mu_j+\lambda} + \sum_{j=s_\lambda+1}^{\infty} \frac{\mu_j}{\mu_j+\lambda}. 
$$
Note that $\sum_{j=1}^{s_\lambda} \mu_j/(\mu_j+\lambda) \leq s_\lambda$, 
then we have
\begin{equation}\label{eq:bound}
s_\lambda \leq h^{-1} \leq s_\lambda + \sum_{j=s_\lambda+1}^{\infty} \frac{\mu_j}{\mu_j+\lambda}\leq s_\lambda + \frac{1}{\lambda}\sum_{j=s_\lambda+1}^{\infty} \mu_j.
\end{equation}
By Assumption 3.3, we have $\sum_{j=s_\lambda+1}^{\infty} \mu_j \leq C s_\lambda \mu_{s_\lambda} \leq s_\lambda \lambda $. Therefore, by (\ref{eq:bound}), we have $h^{-1} \asymp s_\lambda $. Next we show $\sum_{\nu\ge1}(1+\lambda/\mu_\nu)^{-2}\asymp h^{-1}$.  

Note that 
 $$
 \sum_{\nu\ge1}(1+\lambda/\mu_\nu)^{-2} = \sum_{j=1}^{\infty} \frac{\mu_j^2}{(\mu_j+ \lambda)^2} = \sum_{j=1}^{s_\lambda} \big(\frac{\mu_j}{\mu_j+\lambda}\big)^2 + \sum_{j=s_\lambda+1}^{\infty} \big(\frac{\mu_j}{\mu_j+\lambda}\big)^2,
 $$
 similar to (\ref{eq:bound}), we have 
 $$
s_\lambda \leq \sum_{\nu\ge1}(1+\lambda/\mu_\nu)^{-2} \leq s_\lambda + \sum_{j=s_\lambda+1}^{\infty} \big(\frac{\mu_j}{\mu_j+\lambda}\big)^2 \leq s_\lambda + \frac{1}{\lambda^2}\sum_{j=s_\lambda+1}^{\infty} \mu_j^2.
 $$
 Since $\frac{1}{\lambda^2}\sum_{j=s_\lambda+1}^{\infty} \mu_j^2 \leq \frac{\mu_{s_\lambda+1}}{\lambda^2} \sum_{j=s_\lambda+1}^{\infty} \mu_j \leq \frac{1}{\lambda} \sum_{j=s_\lambda+1}^{\infty} \mu_j \leq s_\lambda$. Then we have $\sum_{\nu\ge1}(1+\lambda/\mu_\nu)^{-2} \asymp s_\lambda$. 
 Based on the previous conclusion that $h^{-1}\asymp s_lambda$, we finally get $\sum_{\nu\ge1}(1+\lambda/\mu_\nu)^{-2} \asymp h^{-1}$. 
\end{proof}

\subsection{Verification of Assumption 3.3 }
Let us verify Assumption 3.3 in polynomially decaying kernels (PDK) and exponentially decaying kernels (EDK). 

First consider PDK with $\mu_i \asymp i^{-2m}$ for a constant $m >1/2$
which includes kernels of Sobolev space and Besov Space. An $m$-th order Sobolev space, denoted $\cH^m ([0,1])$, is defined as 
\begin{align*}
\cH^m ([0,1]) = & \{f: [0,1] \to \bbR | f^{(j)} \;\; \text{is abs. cont for} \; j=0, 1,\cdots, m-1, \\
&\text{and}\; f^{m} \in L_2([0,1]) \}.
\end{align*}
An $m$-order periodic Sobolev space, denoted $H_0^m(\mathbb{I})$,
is a proper subspace of $\cH^m([0,1])$ whose element fulfills an additional constraint $g^{(j)}(0)=g^{(j)}(1)$ for $j=0,\ldots,m-1$. 
The basis functions $\varphi_i$'s of $H_0^m(\mathbb{I})$ are
\begin{equation*}
\varphi_i(z)=\left\{ \begin{array}{cc} \sigma,& i=0,\\
                                   \sqrt{2}\sigma\cos(2\pi kz), & i=2k, k=1,2,\ldots,\\
                                   \sqrt{2}\sigma\sin(2\pi kz), & i=2k-1, k=1,2,\ldots.
                 \end{array}\right.
\end{equation*}
The corresponding eigenvalues are $\mu_{2k}=\mu_{2k-1}=\sigma^2(2\pi k)^{-2m}$
for $k\ge 1$ and $\mu_0=\infty$. In this case, $\sup_{i\geq 1} \|\varphi\|_{\sup} <\infty$. For any $k \geq 1$, 
$$
\sum_{i=k+1}^\infty \mu_i \lesssim \int_{k}^\infty x^{-2m}dx = \frac{k^{1-2m}}{2m-1} \lesssim \frac{k\mu_k}{2m-1}.
$$
Therefore, there exists a constant $C< \infty$, such that
 $$\sup_{k\geq 1} \frac{\sum_{i=k+1}^\infty \mu_i}{k\mu_k} = C < \infty.$$
Hence, Assumption 3.3 holds true.

Next, let us consider EDK with $\mu_i \asymp \exp(-\gamma i^p)$ for constants $\gamma>0$ and $p>0$. Gaussian kernel $K(x,x') = \exp\left(-(x-x')^2/\sigma^2\right)$
is an EDK of order $p=2$, with eigenvalues $\mu_i \asymp \exp(-\pi i^2)$ as $i \to \infty$, and the corresponding eigenfunctions
$$
\varphi_i(x) = (\sqrt{5}/4)^{1/4}(2^{i-1}i!)^{-1/2}e^{-(\sqrt{5}-1)x^2/4}H_i((\sqrt{5}/2)^{1/2}x),
$$
where $H_i(\cdot)$ is the $i$-th Hermite polynomial; see \cite{sollich2005understanding} for more details. 
Then $\sup_{i\geq 1} \|\varphi_i\|_{\sup} <\infty$ trivially holds. For any $k\geq 1$,  
$$
\sum_{i=k+1}^\infty \mu_i \lesssim \int_k^\infty e^{-\gamma x^p}dx \\
 = \frac{1}{\gamma p k^{p-1}} e^{-\gamma k^p} - \int_{k}^\infty \frac{p-1}{\gamma p x^{p}} e^{-\gamma x^p} dx \leq \frac{1}{\gamma p k^{p-1}} e^{-\gamma k^p}.  
$$
Therefore, 
$$\sup_{k\geq 1} \frac{\sum_{i=k+1}^\infty \mu_i}{k\mu_k} < \infty.$$
Hence, Assumption 3.3 holds.

\subsection{Proof of Lemma A2}
To prove Lemma A2, based on Lemma 3.5 and Lemma 3.4 in Chapter 21 in \cite{eggermont2009maximum}, we only need to bound 
\begin{align*}
& \|\frac{1}{n}\sum_{i=1}^n K_h(X_i,\cdot) - \E [K_h(X_i), \cdot]\|_{\infty},\\
& \textrm{and} \quad \|\frac{1}{n}\sum_{i=1}^n hK_h'(X_i,\cdot) - h\E [K_h'(X_i), \cdot]\|_{\infty}.
\end{align*}

\begin{lemma}\label{le:add:equivalent_kernel}
Assume that the family $K_h = \sum_{j=1}^d K_{h_0,j}$ with $K_{h_0,j}$, $0 <h_0 \leq 1$ is convolution-like. Then there exists a constant $c$, such that,for all $h$, $0 < h \leq 1$, and for every strictly positive design $X_1, X_2, \cdots, X_n \in (0,1]^d$, 
$$
\|\frac{1}{n}\sum_{i=1}^n K_h (X_i, \cdot)\|_{\infty} \leq c\|\frac{1}{n}\sum_{i=1}^n g_h(X_i - \cdot)\|_{\infty}. 
$$
\end{lemma}
\begin{proof}
For $t=(t_1, \cdots, t_d)\in [0,1]^d$ and $x=(x_1, \cdots, x_d) \in [0,1]^d$, let $S_{nh}(t) = \frac{1}{n}\sum_{i=1}^n K_h (X_i,t)$, and $s^{nh} (t) = \frac{1}{n}\sum_{i=1}^n g_h(X_i -t)$. For $j=1, \cdots, d$,  $K_{h,j}$ satisfies 
$$
K_{h_0,j} (t_j, x_j) = h_0 g_{h,j} (x)K_{h_0,j} (t_j,0) + \int_{0}^1 g_{h_0,j}(x_j-z_j)\{h_0K'_{h_0,j}(t_j,z_j) + K_{h_0,j}(t_j,z_j)\}dz_j,
$$
where $h_0= dh$. Note that $K_{h_0,j}$, $h_0K_{h_0,j}'$ are all convolutional-like, then 
$|h_0K'_{h_0,j}(t_j,z_j)|\leq c h_0^{-1}$ and $|K_{h_0,j}(t_j,z_j)| \leq c h_0^{-1}$. Therefore, 
\begin{align*}
&  \int_{0}^1 g_{h_0,j}(x_j-z_j)\{h_0K'_{h_0,j}(t_j,z_j) + K_{h_0,j}(t_j,z_j)\}dz_j 
\leq  2c\cdot h_0^{-1}\int_0^1 g_{h_0,j}(x_j-z_j) dz_j \\
 = & 2c \cdot h_0^{-2} \int_0^1 e^{-h_0^{-1}(x_j-z_j)}dz_j 
\leq 2c \cdot \big(g_{h_0,j}(x_j) - g_{h_0,j}(x_j-1)\big)
 \leq 2c \cdot g_{h_0,j}(x_j).
\end{align*}
Then, we have 
$K_{h_0,j} (t_j, x_j) \leq h_0 \cdot g_{h_0,j}(x) K_{h_0,j}(t_j,0) + c\cdot g_{h_0,j}(x_j)$. 
\begin{align*}
K_h(x,t) = &  \sum_{j=1}^d K_{h_0,j} (t_j, x_j) \leq h_0 \sum_{j=1}^d g_{h_0,j} (x_j) K_{h_0,j} (t_j, 0) + c \sum_{j=1}^d g_{h_0,j}(x_j)\\
 \leq &  c_1 \sum_{j=1}^d g_{h_0,j}(x_j) + c \sum_{j=1}^d g_{h_0,j}(x_j) \leq c' \sum_{j=1}^d g_{h_0,j} (x_j) = c' g_h(x),
\end{align*}
where $c_1=\max \{h_0K_{h_0,1}(t_1,0), \cdots, h_0K_{h_0,d}(t_d,0)\}$ is a bounded constant by the convolution-like assumption. Let $X_i=x$ and substitute the formula above into the expression for $S_{nh}(t)$ and $s^{nh}(t)$, this gives $S_{nh}(t) \leq c' s^{nh}(0)$. Therefore, 
$\|S_{nh}\|_{\infty} \leq c' |s^{nh}(0)| \leq \|s^{nh}\|_{\infty}$. The last inequality is due to the fact that all $X_i$ are strictly positive, then $s^{nh}(t)$ is continuous at $t=0$, and so $s^{nh}(0)\leq \|s^{nh}\|_\infty$. 

\end{proof}

Let $P_n$ be the empirical distribution function of the design $X_1, X_2, \cdots, X_n$, and let $P_0$ be the design distribution function. Here $P_0= \pi (x)$. Define 
$$
[g_h \circledast (dP_n -dP_0)\big](t)  = \int_{[0,1]^d} g_h(x-t)(dP_n(x) -dP_0(x)),
$$
then based on Lemma \ref{le:add:equivalent_kernel}, we only need to show the following results to prove Lemma A2. 

\begin{lemma}\label{le:emp:bound_g:1}
For all $x=(x_1, \cdots, x_d) \in [0,1]^d, t>0$, 
\begin{equation}\label{le:eq:add:emp}
\PP \Big[ |\big[ g_h \circledast (dP_n -dP_0)\big](x)|>t\Big] \leq 2 \exp\Big\{ -\frac{nht^2}{w_2 + 2/3 t}\Big\},
\end{equation}
where $w_2$ is an upper bound on the density $P_0(x)$.  
\end{lemma}
\begin{proof}
Consider for fixed $x$, $\frac{1}{n} \sum_{i=1}^n g_h (X_i -x) = \sum_{k=1}^d \sum_{i=1}^n \theta_{ik}$, with $\theta_{ik} = \frac{1}{n} g_{h_0,k} (x_{i,k} -x_k)$. Then $\theta_{ik}$ $(i=1,\cdots, n; k=1,\cdots, d)$ are i.i.d. and $|\theta_{ik}|\leq (nh_0)^{-1}$, where $h_0 = d^{-1} h$. For the variance $\Var(\theta_{ik})$, 
\begin{align*}
\Var (\theta_{ik})  = &\frac{1}{n^2} \big\{ [g_{h_0,k}^2 \circledast dP_0](x_k)- ([g_{h_0,k}\circledast dP_0](x))^2\big\}\\
\leq & \frac{1}{n^2} \big[ g_{h_0,k}^2 \circledast dP_0\big](x_k)\\
= & n^{-2}\int_0^1 h_0^{-2} e^{-2h_0^{-1}(X_{ik}-x_k)} dP_0 (x_k)\\
\leq & \frac{1}{2} w_2 n^{-2}h_0^{-1}.
\end{align*}
Therefore, $V:= \sum_{i=1}^n \sum_{k=1}^d \Var (\theta_{ik}) \leq \frac{1}{2} w_2 n^{-1}h^{-1}$. Then by Bernstein's inequality, (\ref{le:eq:add:emp}) has been proved. 
\end{proof}

\begin{lemma}\label{le:emp:bound_g:2}
For all $j=1,\cdots, n$, 
$$
\PP \{[g_h \circledast (dP_n - dP_0)](X_j) > t\} \leq 2 \exp\{-\frac{1/4 nht^2}{w_2 + 2/3 t}\},
$$
provided $t \geq 2(1+w_2) (nh)^{-1}$, where $w_2$ is an upper bound on the density. 
\end{lemma}
\begin{proof}
Consider $j =n$. Note that 
\begin{align*}
[g_h \circledast dP_n](X_n) & = \frac{1}{n}g_h(0) + \frac{1}{n}\sum_{i=1}^{n-1} g_h(X_i -X_n)\\
& = \frac{1}{n}\sum_{k=1}^d g_{h_0,k}(0) + \frac{1}{n} \sum_{i=1}^{n-1}g_h (X_i - X_n)\\
& = d (nh_0)^{-1} + \frac{n-1}{n} [g_h \circledast dP_{n-1}](X_n),
\end{align*}
so that its expectation, conditional on $X_n$, equals
$$
\EE \big[ [g_h \circledast dP_n] (X_n)| X_n\big] = (nh)^{-1} + \frac{n-1}{n} [g_h \circledast dP_0] (X_n).
$$
Then $\PP \big[ |[g_h \circledast (dP_{n-1} - dP_0)](X_n)| > t | X_n \big] \leq 2 \exp\{-\frac{(n-1)ht^2}{w_2 + 2/3 t}\}$. Note that this upper bound does not involve $X_n$, it follows that 
$$
\PP \big[ |[g_h \circledast (dP_{n-1} -d P_0)] (X_n) > t|\big] = \EE \Big[ \PP \big[ |[g_h \circledast (dP_{n-1} - dP_0)](X_n)| > t | X_n \big]\Big]
$$
has the same bound. Finally, note that 
$$
[g_h \circledast (dP_n - dP_0)] (X_n) = \varepsilon_{nh} + \frac{n-1}{n} [g_h \circledast (dP_{n-1}-dP_0)](X_n),
$$
where $|\varepsilon_{nh}| = |(nh)^{-1} - \frac{1}{n}[g_h \circledast dP_0](X_n)| \leq (nh)^{-1} + (nh)^{-1} w_2 \leq c_2  (nh)^{-1}$.  Therefore, 
\begin{align*}
& \PP \Big\{ \big[ | g_h \circledast (dP_n -d P_0)\big](X_n) | > t \Big\}\\
\leq & \PP \Big\{|\big[ g_h \circledast (dP_{n-1} - dP_0)\big](X_n)|> \frac{n}{n-1}(t-c_2 (nh)^{-1}) \Big\}\\
\leq & 2 \exp\Big\{-\frac{nh(t-c_2  (nh)^{-1})^2}{w_2 + 2/3 (t-c_2 (nh)^{-1})}\Big\}.
\end{align*}

\end{proof}

\subsection{Proof of Corollary 3.2}
Note that for any $x,y \in [0,1]^d$, by Lemma A1, we have $K(x,y)\leq c_\varphi^2 h^{-1}$, where $h^{-1}\asymp d \lambda^{-1/(2m)}$, and $\|\mathcal{P}_\lambda f\|^2 \leq \lambda \|f\|^2_{\mathcal{H}} \leq C d\lambda$, then Corollary 3.2 can be easily achieved by applying Theorem 3.1 and Theorem 3.3. 

Next, we show that $d_{N,\lambda, d}^* = d^{\frac{2m+1}{2(4m+1)}}N^{-\frac{2m}{4m+1}}$ is the minimax testing rate. Consider the model
 \begin{equation}\label{eq: normal_sequence}
 \tilde{y} = \theta + w,
 \end{equation}
  where $\theta\in \mathbb{R}^n$ satisfies the ellipse constraint $\sum_{j=1}^n \frac{\theta_j^2}{\mu_j} \leq d$, where $\mu_1\geq \mu_2 \geq \cdots \geq 0 $, and the noise vector $w$ is zero-mean with variance $\frac{\sigma^2}{n}$. Note that model (2.1) is equivalent to  model (\ref{eq: normal_sequence}) (see Example 3 in \cite{wei2017local} for details), thus we only need to prove the minimax testing rate under model (\ref{eq: normal_sequence}) for the testing problem $\theta=0$ with $\mu_j \asymp {\lceil \frac{i}{d}\rceil}^{-2m}$. 

Let $m_u (\delta; \varepsilon): = \argmax_{1\leq k \leq d} \{d\mu_k \geq \frac{1}{2}\delta^2\}$, and $m_l(\delta; \varepsilon) : = \argmax_{1\leq k\leq d} \{d\mu_{k+1} \geq \frac{9}{16}\delta^2\}$. Then by Corollary 1 in \cite{wei2017local}, we have 
$$
\sup\{\delta| \delta \leq \frac{1}{4} \sigma^2 \frac{\sqrt{m_l(\delta; \varepsilon)}}{\delta}\} \leq d_{N,\lambda, d}^* \leq \inf\{\delta| \delta \geq c \sigma^2 \frac{\sqrt{m_u (\delta; \varepsilon)}}{\delta}\}.
$$
Let $\delta^* $ satisfies $\delta^2 \asymp \sqrt{m_l(\delta; \varepsilon)} \asymp \sqrt{m_u(\delta; \varepsilon)}$, we have $\delta^*  = d_{N,\lambda, d}^* \asymp d^{\frac{2m+1}{2(4m+1)}}N^{-\frac{2m}{4m+1}}$.

\bibliographystyle{plainnat}
\bibliography{ref}

\end{document}